\documentclass[draft]{birkjour}
\usepackage[noadjust]{cite}
\usepackage{xcolor}
\usepackage{amssymb}
\RequirePackage[all]{xy}

\usepackage{longtable}

\newtheorem{thm}{Theorem}[section]

\newtheorem{lem}[thm]{Lemma}

\theoremstyle{definition}

\theoremstyle{remark}
\newtheorem{rem}[thm]{Remark}

\newtheorem*{ex}{Example}
\numberwithin{equation}{section}

\newcommand{\BibTeX}{B\kern-0.1emi\kern-0.017emb\kern-0.15em\TeX}
\newcommand{\XYpic}{$\mathrm{X\kern-0.3em\raisebox{-0.18em}{Y}}$-$\mathrm{pic}\,$}

\newcommand{\cl}{C \kern -0.1em \ell}  

\newcommand{\BR}{\mathbb{R}}
\newcommand{\BC}{\mathbb{C}}

\newcommand{\ed}{\end{document}}

\def\F{{\mathbb F}}
\def\C{\mathcal {G}}
\def\Z{{\rm Z}}
\def\Heis{{\rm Heis}}
\def\P{{\rm P}}
\def\H{{\rm H}}
\def\S{{\rm S}}
\def\UT{{\rm UT}}
\def\SUT{{\rm SUT}}
\def\H{{\rm H}}
\def\Aut{{\rm Aut}}
\def\Mat{{\rm Mat}}
\def\GL{{\rm GL}}
\def\ker{{\rm ker}}
\def\rad{{\rm rad}\;}
\def\ad{{\rm ad}}
\def\mod{{\rm \;mod\; }}
\def\Lambd{{\rm\Lambda}}

\begin{document}

%
%
%
%
%
%
%
%
%

\title[On Some Lie Groups in Degenerate Clifford Geometric Algebras]
 {On Some Lie Groups in Degenerate Clifford Geometric Algebras}

\author[E. Filimoshina]{Ekaterina Filimoshina}
\address{%
HSE University\\
Moscow 101000\\
Russia}
\email{filimoshinaek@gmail.com}
\author[D. Shirokov]{Dmitry Shirokov}
\address{%
HSE University\\
Moscow 101000\\
Russia
\medskip}
\address{
and
\medskip}
\address{
Institute for Information Transmission Problems of the Russian Academy of Sciences \\
Moscow 127051 \\
Russia}
\email{dm.shirokov@gmail.com}


\date{\today}
\dedicatory{Last Revised:\\ \today}
\begin{abstract}
In this paper, we introduce and study five families of Lie groups in degenerate Clifford geometric algebras. These Lie groups preserve the even and odd subspaces and some other subspaces under the adjoint representation and the twisted adjoint representation.  
The considered Lie groups contain degenerate spin groups, Lipschitz groups, and Clifford groups as subgroups in the case of arbitrary dimension and signature. The considered Lie groups can be of interest for various applications in physics, engineering, and computer science. 
\end{abstract}
\label{page:firstblob}
\maketitle

\section{Introduction}

In this paper, we study degenerate Clifford geometric algebras $\C_{p,q,r}$ of arbitrary dimension and signature.
Degenerate geometric algebras are important for applications in geometry, computer science, engineering, signal and image processing, physics, etc. 
For instance, projective geometric algebra (PGA) $\C_{p,0,1}$ is useful for computations with flat objects and is applied in computer graphics and vision, robotics, motion capture, dynamics simulations \cite{b1,h3,pga_book,gunn_1,gunn_2}. PGA can be realized as a subalgebra of conformal geometric algebra (CGA) \cite{phys,hestenes_CGA,hd2,pga1,h2}, which has applications in pose estimation, robotics, computer animation, machine learning, neural networks, etc.
\cite{jl1,cs1,la2_,hi1,hd,la1_}. 
The algebras $\C_{3,0,1}$, $\C_{0,3,1}$, even subalgebras $\C^{(0)}_{3,0,1}$ (known as the motor algebra), $\C^{(0)}_{0,3,1}$, $\C^{(0)}_{6,0,2}$, $\C^{(0)}_{6,0,6}$ are applied in robotics and computer vision \cite{b1,bcsd,se}.

We introduce and study five families of Lie groups $\P^{\pm}_{p,q,r}$, $\P_{p,q,r}$, $\P^{\pm\Lambd}_{p,q,r}$, $\P^{\Lambd}_{p,q,r}$, $\P^{\pm\rad}_{p,q,r}$ in the degenerate Clifford geometric algebras $\C_{p,q,r}$ of arbitrary dimension and signature over the field $\F$ of real or complex numbers. These groups preserve the even subspace $\C^{(0)}_{p,q,r}$, the odd subspace $\C^{(1)}_{p,q,r}$, 
the grade-$0$ subspace $\C^0$, the grade-$n$ subspace $\C^{n}_{p,q,r}$, 
and their direct sum under the adjoint representation and the twisted adjoint representation. The twisted adjoint representation has been introduced for the first time in the classical paper \cite{ABS}. It is an important mathematical notion, which is used to describe two-sheeted coverings of orthogonal groups by spin groups. 
The spin groups are subgroups of the well-known Clifford groups and Lipschitz groups \cite{ABS,lg1,lounesto,it4,it5}, which preserve the grade-$1$ subspace under the adjoint and twisted adjoint representations respectively. The degenerate spin groups, Lipschitz and Clifford groups \cite{RA_Z,brooke_1,brooke_2,brooke_3,crum_book,crum, der} are subgroups of the groups introduced in this paper.
This paper generalizes the results of the papers \cite{GenSpin,OnInner}  on Lie groups in the non-degenerate geometric algebras $\C_{p,q}$.

The degenerate spin groups and the other Lie groups considered in this paper can be of interest for various applications in physics, engineering, and computer science. In particular, the spin groups and Lipschitz groups are used in the realization of spinor neural networks \cite{hi1}, spinor image processing \cite{ip}, rotor-based color edge detection \cite{jl1}. The degenerate spin groups are applied in rigid body dynamics \cite{la2_,se}, motion estimation \cite{b2}, and representation theory of Galilei group \cite{brooke_1} in quantum mechanics.  The analogue of the Clifford group, which preserves the Weyl -- Heisenberg group under the adjoint representation, is used in quantum error-correcting codes in quantum computing~\cite{quan}.

In the particular case of the Grassmann algebras $\C_{0,0,1}$ and $\C_{0,0,2}$, the Lie groups $\P^{\pm}_{p,q,r}$, $\P_{p,q,r}$, $\P^{\pm\Lambd}_{p,q,r}$, $\P^{\Lambd}_{p,q,r}$, and $\P^{\pm\rad}_{p,q,r}$ can be realized as subgroups of the groups of invertible upper triangular matrices $\UT(2,\F)$ and $\UT(4,\F)$ respectively. These groups are Borel subgroups of the general linear groups $\GL(2,\F)$ and $\GL(4,\F)$ respectively. The unitriangular group $\SUT(2,\F)$ can be realized as a subgroup of the five considered Lie groups in the case of the Grassmann algebra $\C_{0,0,1}$. In the case of  $\C_{0,0,2}$, the corresponding five Lie groups are closely related to the higher-dimensional Heisenberg group $\Heis_4$ (see, for example, \cite{baker,hall}), which has various applications in quantum mechanics and computing. 
In the case of arbitrary dimension and signature, the relation of the introduced Lie groups  with the well-known matrix Lie groups requires further research. The same problem in the non-degenerate case is solved in the papers \cite{it1,it2,it3}.

The paper is organized as follows. In Section \ref{the_jacobson_radical}, we discuss degenerate Clifford geometric algebras $\C_{p,q,r}$ and prove some auxiliary statements on the Jacobson radical. We present the statements on the adjoint representation and the twisted adjoint representation in Section~\ref{the_kernels}. Section \ref{section_P} introduces the five families of Lie groups in the degenerate geometric algebras $\P^{\pm}_{p,q,r}$, $\P_{p,q,r}$, $\P^{\pm\Lambd}_{p,q,r}$, $\P^{\Lambd}_{p,q,r}$, $\P^{\pm\rad}_{p,q,r}$ and gives their several equivalent definitions. In Section \ref{section_examples}, we provide some examples on the groups $\P^{\pm}_{p,q,r}$, $\P_{p,q,r}$, $\P^{\pm\Lambd}_{p,q,r}$, $\P^{\Lambd}_{p,q,r}$, and $\P^{\pm\rad}_{p,q,r}$ in the cases of the low-dimensional degenerate geometric algebras. In Section \ref{section_gamma}, we prove that the groups $\P^{\pm}_{p,q,r}$, $\P_{p,q,r}$, $\P^{\pm\Lambda}_{p,q,r}$, and $\P^{\Lambda}_{p,q,r}$ preserve the even subspace and the odd subspace under the adjoint and twisted adjoint representations. In Section \ref{section_gamma0n}, we prove that the groups $\P^{\pm}_{p,q,r}$, $\P_{p,q,r}$, and $\P^{\pm\rad}_{p,q,r}$ preserve the grade-$0$ and grade-$n$ subspaces and their direct sum under the adjoint representation and the twisted adjoint representation.  We study the corresponding Lie algebras of the considered Lie groups in Section \ref{lie_alg}. The conclusions follow in Section~\ref{section_conclusions}. We provide a summary of notation used throughout the paper in Appendix \ref{appendix_A}.

This paper is an extended version of the short note (12 pages) in Conference Proceedings \cite{ICACGA} (International Conference of Advanced Computational Applications of Geometric Algebra, Denver, USA, 2022). 
Sections \ref{section_examples} and \ref{section_gamma0n} are new, Section \ref{section_P} is extended (one additional family of Lie groups is introduced and studied). Theorems \ref{eq_ppmrad}, \ref{gogn_pqr}, and \ref{gon_pqr} are presented for the first time. The detailed proofs of Lemmas \ref{lemma_XVVX_gr1}--\ref{modd^XU=UX} and Theorem \ref{eq_P_l2} are presented for the first time.

\section{Degenerate geometric algebra and the Jacobson radical}\label{the_jacobson_radical}

Let us consider the (Clifford) geometric algebra  \cite{hestenes,lounesto,p}   $\C(V)=\C_{p,q,r}$\label{def_ga}, $p+q+r=n\geq1$, over a vector space $V$ with a symmetric bilinear form $g$.
We consider the real case $V=\BR^{p,q,r}$ and the complex case $V=\BC^{p+q,r}$. 
We use $\F$\label{def_F} to denote the field of real numbers $\BR$ in the first case and the field of complex numbers $\BC$ in the second case respectively. In this paper, we concentrate on the degenerate geometric algebras with $r\neq0$, but all the following statements are true for arbitrary $r\geq 0$.

We denote the identity element of the algebra $\C_{p,q,r}$ by $e$, the generators by $e_a$, $a=1,\ldots,n$. In the case of the real geometric algebra $\C(\BR^{p,q,r})$, the generators satisfy
\begin{eqnarray}\label{eq_clalg}
e_a e_b + e_b e_a =2 \eta_{ab} e,\qquad  a,b=1,\ldots,n,
\end{eqnarray}
where $\eta=(\eta_{ab})$ is the diagonal matrix with $p$ times $1$, $q$ times $-1$ and $r$ times $0$ on the diagonal. In the case of the complex geometric algebra $\C(\BC^{p+q,r})$, the generators satisfy the same conditions but with the diagonal matrix $\eta$ with $p+q$ times $1$ and $r$ times $0$ on the diagonal.
Let us denote by $\Lambd_r:=\C_{0,0,r}$\label{def_grass} the subalgebra of $\C_{p,q,r}$, which is the Grassmann (exterior) algebra \cite{crum_book,phys,lounesto}.

Consider the subspaces $\C^{k}_{p,q,r}$\label{def_gk} of grades $k=0,1,\ldots,n$, which elements are linear combinations of the basis elements $e_{a_1\ldots a_k}:=e_{a_1}\cdots e_{a_k}$, $a_1<\cdots<a_k$, with ordered multi-indices of length $k$.  
Note that the subspace $\C^{0}_{p,q,r}$ of grade $0$ does not depend on the signature of the algebra, so we denote it by $\C^{0}$\label{def_g0} without the lower indices $p,q,r$.
We use the upper multi-index instead of the direct sum symbol in order to denote the direct sum of different subspaces. For example, $\C^{0n}_{p,q,r}:=\C^{0}\oplus\C^{n}_{p,q,r}$\label{def_g0n}.

The grade involute of the element $U\in\C_{p,q,r}$ is denoted by $\widehat{U}$\label{def_grade_inv}. This operation has the following well-known property: $\widehat{UV}=\widehat{U}\widehat{V}$ for any $U, V\in \C_{p,q,r}$.
Consider the even $\C^{(0)}_{p,q,r}$ and odd $\C^{(1)}_{p,q,r}$ subspaces:
\begin{eqnarray}\label{even_odd_subspaces}
\C^{(k)}_{p,q,r}=\{U\in\C_{p,q,r}:\quad \widehat{U}=(-1)^k U\}=\bigoplus_{j=k \mod 2} \C^j_{p,q,r},\qquad k=0, 1
\end{eqnarray}
with the property
\begin{eqnarray}\label{even}
\C^{(k)}_{p,q,r}\C^{(l)}_{p,q,r}\subset\C^{(k+l)\mod{2}}_{p,q,r},\qquad k,l=0,1.
\end{eqnarray}

Let us consider the Jacobson radical $\rad\C_{p,q,r}$\label{def_rad} of the algebra $\C_{p,q,r}$. 
Let $A,B,C$ be ordered multi-indices with the non-zero length and $e_A=e_{a_1}\cdots e_{a_k}$ with $\{a_1,\ldots,a_k\}\subseteq\{1,\ldots,p\}$, $e_B=e_{b_1}\cdots e_{b_l}$ with $\{b_1,\ldots,b_l\}\subseteq\{p+1,\ldots,p+q\}$, $e_C=e_{c_1}\cdots e_{c_m}$ with $\{c_1,\ldots,c_m\}\subseteq\{p+q+1,\ldots,n\}$.
An arbitrary element $y\in\rad\C_{p,q,r}$ has the form
\begin{eqnarray}\label{arb_rad}
y=\sum_C v_C e_C +\sum_{A,C} v_{AC} e_A e_C+\sum_{B,C} v_{BC} e_B e_C +\sum_{A,B,C} v_{ABC}e_{A}e_{B}e_{C},
\end{eqnarray}
where $v_C,v_{AC},v_{BC},v_{ABC}\in\F$.
\begin{rem}\label{jac_inv}
Any element of the Jacobson radical is non-invertible (see \cite{JR_1}).
\end{rem}
The Jacobson radical of the Grassmann algebra $\C_{0,0,n}=\Lambd_n$ is the direct sum of the subspaces of grades $1,\ldots, n$:
\begin{eqnarray*}
\rad\C_{0,0,n}=\C^{1}_{0,0,n}\oplus\C^{2}_{0,0,n}\oplus\cdots\oplus\C^{n}_{0,0,n},\quad
\C_{0,0,n}=\C^{0}\oplus\rad\C_{0,0,n}.
\end{eqnarray*}
The non-degenerate algebra $\C_{p,q,0}$ is semi-simple and $\rad\C_{p,q,0}=\{0\}$ (\cite{RA_Z,crum_book,JR_1}).

We need the following well-known (see, for example, \cite{RA_2,JR_1}) lemma.
\begin{lem}\label{well_known_lemma}
The element $e+x y$ is invertible for any $y\in\rad\C_{p,q,r}$, $x\in\C_{p,q,r}$.
\end{lem}
The subset of invertible elements of any set is denoted with $\times$\label{def_A_inv}. For example, we denote the group of invertible elements of the algebra $\C_{p,q,r}$ by $\C^{\times}_{p,q,r}$.

\begin{lem}\label{lemma_times_rad}
The element $T\in\C^{0}\oplus\rad\C_{p,q,r}$ is invertible if and only if its projection onto the subspace of grade $0$ is non-zero:
\begin{eqnarray*}
T\in\C^{0}\oplus\rad\C_{p,q,r},\quad\langle T\rangle_0\neq0\qquad\Leftrightarrow\qquad T\in(\C^{0}\oplus\rad\C_{p,q,r})^{\times}.
\end{eqnarray*}
\end{lem}
\begin{proof}
Suppose $\langle T\rangle_0\neq 0$ for some  $T=\alpha e + W=\alpha(e+\frac{1}{\alpha}W)$, where $\alpha\in\F^{\times}$, $W\in\rad\C_{p,q,r}$. We have $e+\frac{1}{\alpha}W\in\C^{\times}_{p,q,r}$ by Lemma \ref{well_known_lemma}; thus, $T\in(\C^{0}\oplus\rad\C_{p,q,r})^{\times}$.
 
Suppose $T\in(\C^{0}\oplus\rad\C_{p,q,r})^{\times}$. 
Assume $\langle T\rangle_0 = 0$; then $T\in\rad\C_{p,q,r}^{\times}$, and we get a contradiction by Remark \ref{jac_inv}.
\end{proof}

\begin{rem}\label{inv_gr0n}
The inverse of any invertible $T=\alpha e +\beta e_{1\ldots n}\in(\C^{0}\oplus\C^{n}_{p,q,r})^{\times}$, where $\alpha,\beta\in\F$, has the form $T^{-1}=\alpha e - \beta e_{1\ldots n}\in(\C^{0}\oplus\C^{n}_{p,q,r})^{\times}$, since $(\alpha e +\beta e_{1\ldots n})(\alpha e -\beta e_{1\ldots n})=\alpha^2 e -\beta^{2}(e_{1\ldots n})^2\in\C^{0}$.
\end{rem}

\section{Adjoint and twisted adjoint representations in $\C_{p,q,r}$}\label{the_kernels}

Consider the adjoint representation $\ad$ acting on the group of all invertible elements
$\ad:\C^{\times}_{p,q,r} \to \Aut\C_{p,q,r}$ as $T \mapsto \ad_T$, where 
\begin{eqnarray}\label{ar}
\ad_{T}(U)=TU T^{-1},\qquad U\in\C_{p,q,r},\qquad T\in\C^{\times}_{p,q,r}.
\end{eqnarray}
Consider the twisted adjoint representation \cite{ABS} $\check{\ad}$ acting on the group of all invertible elements $\check{\ad}:\C^{\times}_{p,q,r} \to \Aut\C_{p,q,r}$ as $T \mapsto \check{\ad}_T$, where 
\begin{eqnarray}
\check{\ad}_{T}(U)=\widehat{T}U T^{-1},\qquad U\in\C_{p,q,r}^1,\qquad T\in\C^{\times}_{p,q,r}.\label{twa}
\end{eqnarray}
The formula (\ref{twa}) defines the twisted adjoint representation only for the elements of the grade-$1$ subspace (vectors). There are two ways to define it for the case of other grades. The first one is to define it by the same formula \cite{Harvey,LuSv,Dai,Choi,Chr}:
\begin{eqnarray}
\check{\ad}_{T}(U)=\widehat{T}U T^{-1},\qquad U\in\C_{p,q,r},\qquad T\in\C^{\times}_{p,q,r}.\label{twa1}
\end{eqnarray}
The second way is to define it with different signs \cite{Knus,Helm,Wal}:
\begin{eqnarray}
\tilde{\ad}_{T}(U)=(-1)^{km}TU T^{-1},\qquad U\in\C^k_{p,q,r},\qquad T\in\C^{m\times}_{p,q,r},\label{twa2}
\end{eqnarray}
where we use another notation $\tilde{\ad}_{T}$ so as not to confuse this operation with the operation (\ref{twa1}). Extending this operation by the linearity $\tilde{\ad}_{T}(U+V)=\tilde{\ad}_{T}(U)+\tilde{\ad}_{T}(V)$, we finally get
\begin{eqnarray}
&&\tilde{\ad}_{T}(U)=TU_0 T^{-1}+\widehat{T} U_1 T^{-1},\qquad T\in\C^{\times}_{p,q,r},\label{twa22}\\
&&U=U_0+U_1\in\C_{p,q,r},\qquad U_0\in\C^{(0)}_{p,q,r},\qquad U_1\in\C^{(1)}_{p,q,r}.\nonumber
\end{eqnarray}
Note that
\begin{eqnarray}
&&\tilde{\ad}_{T}(U_0)=\ad_T(U_0),\qquad \forall U_0\in\C_{p,q,r}^{(0)},\qquad T\in\C^{\times}_{p,q,r},\label{ad_t_1}\\ 
&&\tilde{\ad}_{T}(U_1)=\check{\ad}_T(U_1),\qquad \forall U_1\in\C_{p,q,r}^{(1)},\qquad T\in\C^{\times}_{p,q,r}.\label{ad_t_2}
\end{eqnarray}
Each of the two ways to define a twisted adjoint representation has its own advantages, which are indicated in the works cited above. The first way (\ref{twa1}) is also related to similar operations, which are considered in the representation theory of Lie groups, see, for example, \cite{Zer}. The second way (\ref{twa2}) is preferable for obtaining the correct signs when we consider reflections of elements of higher grades with respect to hyperplanes in geometric algebras using the Cartan--Dieudonn\'e theorem. In this paper, for the convenience of readers, we give answers to all the posed questions for both operations $\check{\ad}_{T}$ and $\tilde{\ad}_{T}$.

\begin{lem}\label{lemma_XVVX_gr1}
We have
\begin{eqnarray*}
\{X\in\C_{p,q,r}:\quad \widehat{X} V = V X\quad \forall V\in\C^{1}_{p,q,r}\}=\Lambd_r.
\end{eqnarray*}
\end{lem}
\begin{proof}
Let us prove that the right set is a subset of the left one.
Suppose $X\in\Lambda_r$ has the following decomposition over a basis: $X=X_1+\cdots +X_k$, $X_i=\alpha_i e_{A_i}$, $\alpha_i\in\F$, $i=1,\ldots, k$, $A_i\neq A_j$ for any $i\neq j$.  We have $e_a X_i = X_i e_a = 0$ if $A_i$ contains $a$. If $X_i\in\Lambda^{(0)}_r$ does not contain $a$, then $e_a X_i = X_i e_a$. If $X_i\in\Lambda^{(1)}_r$ does not contain $a$, then $e_a X_i = - X_i e_a$. Thus, $V X = \widehat{X} V$ for any $V\in\C^{1}_{p,q,r}$ by linearity.

Let us prove that the left set is a subset of the right one. Suppose $V=e_a$, $\forall a=1,\ldots, n$; then $\widehat{X} e_a = e_a X$. Let us represent $X$ in the form $X=A_0 + A_1 + e_a B_0 + e_a B_1$, where $A_0,B_0\in\C^{(0)}_{p,q,r}$, $A_1,B_1\in\C^{(1)}_{p,q,r}$ and $A_0,B_0,A_1,B_1$ do not contain $e_a$. We have $(A_0 + A_1 + e_a B_0 + e_a B_1)\widehat{\;\;} e_a = e_a (A_0 + A_1 + e_a B_0 + e_a B_1)$, i.e. $A_0 e_a - A_1 e_a - e_a B_0 e_a + e_a B_1 e_a = e_a A_0 + e_a A_1 + (e_a)^2 B_0 + (e_a)^2 B_1$. Since $A_0 e_a = e_a A_0$, $A_1 e_a = -e_a A_1$, $B_0 e_a = e_a B_0$, $B_1 e_a = - e_a B_1$, we obtain $(e_a)^2 B_0 + (e_a)^2 B_1 = 0$, i.e. $(e_a)^2 B_0 = 0$ and $(e_a)^2 B_1 = 0$. If $(e_a)^2 = 0$, then $B_0$ and $B_1$ are any elements from $\C_{p,q,r}$. If $(e_a)^2\neq 0$, then $B_0=B_1=0$; therefore, $X$ does not contain $e_a$. Acting similarly to all generators $e_a$, $a=1,\ldots, n$, we obtain that $X$ does not contain invertible generators. Thus, $X\in\Lambda_r$.
\end{proof}

It is well-known (see, for example,  \cite{RA_Z,brooke_3}) that 
the center of $\C_{p,q,r}$ is
\begin{eqnarray}\label{Zpqr}
\Z_{p,q,r}=
\left\lbrace
\begin{array}{lll}
\Lambd^{(0)}_{r}\oplus\C^n_{p,q,r}&&\quad\mbox{if $n$ is odd},
\\
\Lambd^{(0)}_{r}&&\quad\mbox{if $n$ is even}.
\end{array}
\right.
\end{eqnarray}

\begin{lem}\label{lemma_XVVX}
We have
\begin{eqnarray*}
\!\!\!\!\!\!\!\!\!\!\!\!&&\{X\in\C_{p,q,r}:\; XV=VX \;\; \forall V\in\C^{(0)}_{p,q,r}\}=
\left\lbrace
\begin{array}{lll}\label{XVVX_r}
\Lambd_r\oplus\C^n_{p,q,r},&& r\neq n,
\\
\Lambd_n,&& r=n,
\end{array}
\right.
\\
\!\!\!\!\!\!\!\!\!\!\!\!&&\{X\in\C_{p,q,r}:\; \widehat{X}V=VX \;\;\forall V\in\C^{(0)}_{p,q,r}\}=
\left\lbrace
\begin{array}{lll}\label{XVVX_r_check}
\!\!\Lambd^{(0)}_r,\quad\mbox{$n$ is odd};\; \mbox{$r=n$ is even},
\\
\!\!\Lambd^{(0)}_r\oplus\C^n_{p,q,r},\quad\mbox{$n$ is even},\; r\neq n.
\end{array}
\right.
\end{eqnarray*}
\end{lem}
\begin{proof}
Let us prove that the set $\Lambd_r\oplus\C^n_{p,q,r}$ in the case $r\neq n$ or the set $\Lambda_n$ in the case $r=n$ is a subset of the set $\{X\in\C_{p,q,r}:\; XV=VX \;\; \forall V\in\C^{(0)}_{p,q,r}\}$. Suppose $X=Y+H$, where $Y\in\Lambda_r$,  $H\in\C^{n}_{p,q,r}$ if $r\neq n$ and $H=0$ if $r=n$. We get $X e_{ab}=(Y+H)e_{ab}=e_a\widehat{Y}e_b+e_{ab}H=e_{ab}(Y+H)=e_{ab}X$, where we use Lemma \ref{lemma_XVVX_gr1} and that $H\in\Z_{p,q,r}$ if $n$ is odd, $H$ commutes with all even elements if $n$ is even. Since we can get any even basis element as a product of grade-$2$ elements, $XV=VX$ for any $V\in\C^{(0)}_{p,q,r}$ by linearity.

Let us prove that the set $\{X\in\C_{p,q,r}:\; XV=VX \;\; \forall V\in\C^{(0)}_{p,q,r}\}$ is a subset of the set $\Lambd_r\oplus\C^n_{p,q,r}$ if $r\neq n$ and a subset of the set $\Lambda_n$ if $r=n$. Suppose $V=e_{ab}$, $\forall a<b$; then $Xe_{ab}=e_{ab}X$. Consider the case $a=1$, $b=2$. Let us represent $X$ in the form $X= A + e_1 B + e_2 C + e_{12} D$, where $A,B,C,D$ contain neither $e_1$ nor $e_2$. We have $(A + e_1 B + e_2 C + e_{12} D) e_{12} = e_{12} (A + e_1 B + e_2 C + e_{12} D)$. Since $A e_{12} = e_{12} A$, $D e_{12}= e_{12} D$, $e_1 B e_{12} = - e_{12} e_1 B$, $e_2 C e_{12}=- e_{12} e_2 C$, we obtain 
\begin{eqnarray}\label{e1e2_r}
(e_1)^2 e_2 B- (e_2)^2 e_1 C=0.
\end{eqnarray}
If $(e_1)^2\neq0$ and $(e_2)^2\neq0$ (case $1$), then from (\ref{e1e2_r}) it follows that $B=C=0$, since $B$ and $C$ contain neither $e_1$ nor $e_2$. If $(e_1)^2\neq0$ and $(e_2)^2=0$ (case $2$), then from (\ref{e1e2_r}) it follows that  $C$ is any element from $\C_{p,q,r}$ and $e_2 B=0$. Hence, $B=0$, since $B$ does not contain $e_2$. If $(e_1)^2=0$ and $(e_2)^2=0$ (case $3$), then from (\ref{e1e2_r}) it follows that $B,C$ are any elements from $\C_{p,q,r}$. Acting similarly to all other $a<b$, we obtain that
each of the summands of $X$ either does not contain any invertible generators or contains all invertible generators (it follows from case $1$); there is no such summand in $X$ that contains at least one invertible generator but does not contain any of non-invertible generators (it follows from case $2$). Thus, $X\in\Lambda_r\oplus\C^n_{p,q,r}$ if $r\neq n$ and $X\in\Lambda_n$ if $r= n$.

Let us prove that $\{X\in\C_{p,q,r}:\; \widehat{X}V=VX \;\;\forall V\in\C^{(0)}_{p,q,r}\}$ coincides with $\Lambd^{(0)}_r$ if $n$ is odd or $r=n$ is even and with $\Lambd^{(0)}_r\oplus\C^n_{p,q,r}$ if $n$ is even and $r\neq n$. Substituting $V=e\in\C^{(0)}_{p,q,r}$ into $\widehat{X} V = V X$, we obtain $\widehat{X}=X$. Therefore, $X\in\C^{(0)}_{p,q,r}$. We get
\begin{eqnarray*}
&&\{X\in\C_{p,q,r}: \quad \widehat{X}V=VX \quad \forall V\in\C^{(0)}_{p,q,r}\}
\\
&&=\{X\in\C^{(0)}_{p,q,r}:\quad XV=VX \quad \forall V\in\C^{(0)}_{p,q,r}\}
\\
&&=\{X\in\C_{p,q,r}:\quad XV=VX \quad \forall V\in\C^{(0)}_{p,q,r}\}\cap\C^{(0)}_{p,q,r}
\\
&&
=
\left\lbrace
\begin{array}{lll}
(\Lambda_r\oplus\C^n_{p,q,r})\cap\C^{(0)}_{p,q,r},&&r\neq n,
\\
\Lambda_n\cap\Lambda^{(0)}_{n},&&r=n,
\end{array}
\right.
\end{eqnarray*}
and the proof is completed.
\end{proof}

\begin{lem}\label{modd^XU=UX}
Consider an arbitrary element $X\in\C_{p,q,r}$ and an arbitrary fixed subset $\H$ of the set $\C^{(0)}_{p,q,r}\cup\C^{(1)}_{p,q,r}$. If $\widehat{X}U=U X$ for any $U\in \H$,
then we have
\begin{eqnarray}
\widehat{X}(U_1 \cdots U_m)= (U_1 \cdots U_m) X\qquad \forall U_1,\ldots,U_m\in \H
\end{eqnarray}
for any odd natural number $m$.
\end{lem}
\begin{proof}
The proof is word for word the same as the proof of this statement in the particular case $r=0$, i.e. in the case of the non-degenerate algebra $\C_{p,q,0}$ (see Lemma 4 \cite{GenSpin}). 
\end{proof}

Let us consider the kernels of the adjoint and the twisted adjoint representations:
\begin{eqnarray*}
&&\ker{(\ad)}=\{T\in\C^{\times}_{p,q,r}:\quad T U T^{-1}=U,\quad \forall U \in\C_{p,q,r}\},
\\
&&\ker(\check{\ad})=\{T\in\C^{\times}_{p,q,r}:\quad \widehat{T}UT^{-1}=U,\quad \forall U\in\C_{p,q,r}\},
\\
&&\ker(\tilde{\ad})=\{T\in\C^{\times}_{p,q,r}:\quad T U_0 T^{-1} +\widehat{T}U_1T^{-1}=U,
\\
&&\quad\quad\quad\quad\quad\quad\forall U=U_0+U_1\in\C_{p,q,r},\quad U_0\in\C^{(0)}_{p,q,r},\quad U_1\in\C^{(1)}_{p,q,r}\}.
\end{eqnarray*}

\begin{lem}\label{lemma_3}
We have
\begin{eqnarray}\label{ker_ad_lemma}
&&\ker(\ad)=\Z^{\times}_{p,q,r}=
\left\lbrace
\begin{array}{lll}
(\Lambd^{(0)}_r\oplus\C^n_{p,q,r})^{\times}&&\mbox{if $n$ is odd},
\\
\Lambd^{(0)\times}_r&&\mbox{if $n$ is even},
\end{array}
\right.
\\ 
&&\ker(\check{\ad})=\Lambd^{(0)\times}_r,\label{lemma_kerchad}
\\
&&\ker(\tilde{\ad})=\Lambda^{\times}_r.\label{lemma_kertad}
\end{eqnarray}
\end{lem}
\begin{proof}
We obtain (\ref{ker_ad_lemma}) from  (\ref{Zpqr}).

Let us prove $\Lambd^{(0)\times}_r\subseteq\ker{(\check{\ad})}$. Suppose $T\in\Lambd^{(0)\times}_r$; then $T U T^{-1}=U$ for any $U\in\C_{p,q,r}$. Since $T$ is even, we have $\widehat{T}=T$; therefore, $\widehat{T} U T^{-1}=U$ for any $U\in\C_{p,q,r}$.

Let us prove $\ker{(\check{\ad})}\subseteq\Lambd^{(0)\times}_r$. Suppose $T\in\C^{\times}_{p,q,r}$ satisfies $\widehat{T}UT^{-1}=U$ for any $U\in\C$. Substituting the element $U=e$, we obtain $\widehat{T}=T$; hence, $T\in\C^{(0)\times}_{p,q,r}$ and $T U T^{-1}=U$ for any $U\in \C_{p,q,r}$. In other words, $T\in\C^{(0)\times}_{p,q,r}\cap\ker{(\ad)}$. Using  (\ref{ker_ad_lemma}), we obtain $T\in\C^{(0)\times}_{p,q,r}\cap(\Lambd^{(0)}_r\oplus\C^{n}_{p,q,r})^{\times}=\Lambd^{(0)\times}_r$ in the case of odd $n$, $T\in\Lambd^{(0)\times}_r$ in the case of even $n$, and the proof is completed.

Let us prove $\ker(\tilde{\ad})\subseteq\Lambda^{\times}_r$. Suppose $T\in\ker(\tilde{\ad})$; then $\widehat{T}U_1T^{-1}=U_1$ for any $U_1\in\C^{1}_{p,q,r}$. Thus, $T\in\Lambda^{\times}_r$ by Lemma \ref{lemma_XVVX_gr1}.

Now we must only prove that $\Lambda^{\times}_r\subseteq\ker(\tilde{\ad})$. Suppose $T\in\Lambda^{\times}_r$; then $T\C^{(0)}_{p,q,r}=\C^{(0)}_{p,q,r}T$ and $\widehat{T}\C^{1}_{p,q,r}=\C^{1}_{p,q,r}T$ by Lemmas \ref{lemma_XVVX} and \ref{lemma_XVVX_gr1} respectively. Since any odd basis element can be represented as a product of an odd number of generators, we obtain $\widehat{T}\C^{(1)}_{p,q,r}=\C^{(1)}_{p,q,r}T$ by Lemma \ref{modd^XU=UX}. Thus, $T U_0 T^{-1} +\widehat{T}U_1T^{-1}=U$ for all $U=U_0+U_1\in\C_{p,q,r}$, where $U_0\in\C^{(0)}_{p,q,r}$ and $U_1\in\C^{(1)}_{p,q,r}$, and the proof is completed.
\end{proof}

In the particular case of the non-degenerate algebra $\C_{p,q,0}$, we get the well-known statement
\begin{eqnarray*}
\ker(\check{\ad})=\ker(\tilde{\ad})=\C^{0\times},\qquad
\ker(\ad)=
\left\lbrace
\begin{array}{lll}
(\C^{0}\oplus\C^{n}_{p,q,0})^{\times}&\quad&\mbox{if $n$ is odd},
\\
\C^{0\times}&\quad&\mbox{if $n$ is even}.
\end{array}
\right.
\end{eqnarray*}
In the particular case of the Grassmann algebra $\C_{0,0,n}=\Lambd_n$, we obtain
\begin{eqnarray*}
\ker(\check{\ad})=\Lambd^{(0)\times}_{n},\quad \ker(\tilde{\ad})=\Lambda^{\times}_{n},\quad
\ker(\ad)=
\left\lbrace
\begin{array}{lll}
(\Lambd^{(0)}_{n}\oplus\C^{n}_{0,0,n})^{\times}&&\!\!\!\!\!\!\mbox{if $n$ is odd},
\\
\Lambd^{(0)\times}_{n}&&\!\!\!\!\!\!\mbox{if $n$ is even}.
\end{array}
\right.
\end{eqnarray*}

\section{The groups $\P^{\pm}_{p,q,r}$, $\P_{p,q,r}$, $\P^{\pm\Lambd}_{p,q,r}$, $\P^{\Lambd}_{p,q,r}$, and $\P^{\pm\rad}_{p,q,r}$}\label{section_P}
Let us denote by $\S_{p,q,r}$ the following subset of the center $\Z_{p,q,r}$ (\ref{Zpqr}):
\begin{eqnarray}
\S_{p,q,r}:=\left\lbrace
\begin{array}{lll}\label{S}
\C^{0}\oplus\C^n_{p,q,r}&&\quad\mbox{if $n$ is odd},
\\
\C^{0}&&\quad\mbox{if $n$ is even}.
\end{array}
\right.
\end{eqnarray}
Note that $\S_{p,q,r}\oplus(\Lambd^{(0)}_r\setminus\C^{0})=\Z_{p,q,r}$. In the case of the non-degenerate algebra $\C_{p,q,0}$, we have $\S_{p,q,0}=\Z_{p,q,0}$.

Let us consider the groups  $\P^{\pm}_{p,q,r}$ and $\P_{p,q,r}$:
\begin{eqnarray}
\P^{\pm}_{p,q,r}&:=&\C^{(0)\times}_{p,q,r}\cup\C^{(1)\times}_{p,q,r},\label{P+-}
\\
\P_{p,q,r}&:=&\P^{\pm}_{p,q,r}\Z^{\times}_{p,q,r}\label{defpz}
\\
&&\quad=
\left\lbrace
\begin{array}{lll}\label{P_}
(\C^{(0)\times}_{p,q,r}\cup\C^{(1)\times}_{p,q,r})(\Lambd^{(0)}_r\oplus\C^{n}_{p,q,r})^{\times},&&\mbox{$n$ is odd},
\\
(\C^{(0)\times}_{p,q,r}\cup\C^{(1)\times}_{p,q,r})\Lambd^{(0)\times}_r,&&\mbox{$n$ is even},
\end{array}
\right.
\\
&=&\P^{\pm}_{p,q,r}\S_{p,q,r}^{\times}
\\
&&\quad=
\left\lbrace
\begin{array}{lll}\label{P__0_}
(\C^{(0)\times}_{p,q,r}\cup\C^{(1)\times}_{p,q,r})(\C^{0}\oplus\C^{n}_{p,q,r})^{\times},&&\mbox{$n$ is odd},
\\
\C^{(0)\times}_{p,q,r}\cup\C^{(1)\times}_{p,q,r},&&\mbox{$n$ is even},
\end{array}
\right.
\end{eqnarray}
where we get (\ref{P__0_}) by Lemma \ref{rem_eq_P_easy}. 
Note that $\Z^{\times}_{p,q,r}=\ker{(\ad)}$ in (\ref{defpz})--(\ref{P_}) by Lemma \ref{lemma_3}.
In the particular case $\C_{p,q,0}$, we obtain the groups from the paper \cite{OnInner}:
\begin{eqnarray}\label{ppp1}
\P^{\pm}_{p,q,0}=\P^{\pm}=\C^{(0)\times}_{p,q,0}\cup\C^{(1)\times}_{p,q,0},\qquad \P_{p,q,0}=\P=\Z^{\times}_{p,q,0}(\C^{(0)\times}_{p,q,0}\cup\C^{(1)\times}_{p,q,0}).
\end{eqnarray}

\begin{lem}\label{rem_eq_P_easy}
In the case of arbitrary $\C_{p,q,r}$, we have
\begin{eqnarray}
\!\!\!\!\!(\C^{(0)\times}_{p,q,r}\cup\C^{(1)\times}_{p,q,r})\Lambd^{(0)\times}_r=\C^{(0)\times}_{p,q,r}\cup\C^{(1)\times}_{p,q,r}.\label{rem_eq_P_easy_0}
\end{eqnarray}
If $r=n$ is even, then
\begin{eqnarray}
\!\!\!\!\!\Lambd^{(0)\times}_n=\Lambd^{(0)\times}_n(\C^{0}\oplus\Lambda^{n}_{n})^{\times}.\label{rem_eq_P_easy_1_0}
\end{eqnarray}
If $n$ is even and $r\neq n$ or $n$ is odd, then 
\begin{eqnarray}
\!\!\!\!\!(\C^{(0)\times}_{p,q,r}\cup\C^{(1)\times}_{p,q,r})(\Lambd^{(0)}_r\oplus\C^{n}_{p,q,r})^{\times}=(\C^{(0)\times}_{p,q,r}\cup\C^{(1)\times}_{p,q,r})(\C^{0}\oplus\C^{n}_{p,q,r})^{\times}.\label{rem_eq_P_easy_1}
\end{eqnarray}
\end{lem}
\begin{proof}
The statement (\ref{rem_eq_P_easy_0}) is true by (\ref{even}). 
The proof of the statement (\ref{rem_eq_P_easy_1}) in the case $r=0$ is trivial, since $\Lambd^{(0)}_0=\C^{0}$. 
Consider the case $r\neq0$. 
The right sets in (\ref{rem_eq_P_easy_1_0}) and (\ref{rem_eq_P_easy_1}) are subsets of the corresponding left sets. 
Let us prove that the left set in (\ref{rem_eq_P_easy_1_0}) is a subset of the right one in the case $r=n$ is even and that
the left set in (\ref{rem_eq_P_easy_1}) is a subset of the right one in the case $n$ is even and $r\neq n$ or $n$ is odd. Suppose $T=A W$, where $A=e$ if $r=n$ is even, $A\in\C^{(0)\times}_{p,q,r}\cup\C^{(1)\times}_{p,q,r}$ if $n$ is even and $r\neq n$ or $n$ is odd, and $W=\alpha e + X +\beta e_{1\ldots n}$ in the case of any $n,r$, where $\alpha,\beta\in\F$ and $X\in\Lambd^{(0)}_r\setminus(\C^{0}\oplus\C^{n}_{p,q,r})$. 
Since $W$ is invertible, $\alpha\neq0$ by Lemma \ref{lemma_times_rad}. Then we get $W=(e + \frac{1}{\alpha}X)(\alpha e +\beta e_{1\ldots n})\in\Lambd^{(0)\times}_r(\C^{0}\oplus\C^{n}_{p,q,r})^{\times}$, where the first factor is invertible by Lemma \ref{lemma_times_rad}. 
Hence,
$T=A W=A(e + \frac{1}{\alpha}X)(\alpha e +\beta e_{1\ldots n})$.
So, $T\in\Lambd^{(0)\times}_n(\C^{0}\oplus\Lambda^{n}_{n})^{\times}$ in the case $r=n$ is even and $T\in(\C^{(0)\times}_{p,q,r}\cup\C^{(1)\times}_{p,q,r})\Lambd^{(0)\times}_r(\C^{0}\oplus\C^{n}_{p,q,r})^{\times}=(\C^{(0)\times}_{p,q,r}\cup\C^{(1)\times}_{p,q,r})(\C^{0}\oplus\C^{n}_{p,q,r})^{\times}$ in the other cases, and the proof is completed.
\end{proof}

Also let us consider the groups $\P^{\pm\Lambd}_{p,q,r}$ and $\P^{\Lambd}_{p,q,r}$:
\begin{eqnarray}
\!\!\!\!\!\!\!\!\P^{\pm\Lambd}_{p,q,r}
\!&:=&\!\Lambd^{\times}_r\P^{\pm}_{p,q,r}=\P^{\pm}_{p,q,r}\Lambd^{\times}_r=(\C^{(0)\times}_{p,q,r}\cup\C^{(1)\times}_{p,q,r})\Lambd^{\times}_r,\label{P_Lambd}
\\
\!\!\!\!\!\!\!\!\P^{\Lambd}_{p,q,r}\!&:=&\!\Lambd^{\times}_r \P_{p,q,r}=\P_{p,q,r}\Lambd^{\times}_r\label{pLambd=pLambd}
\\
\!&=&\!\P^{\pm\Lambd}_{p,q,r}\Z^{\times}_{p,q,r}
\\
&&\!\!\!\!\!\!\!\!\!\!\!\!\quad=
\left\lbrace
\begin{array}{lll}
(\C^{(0)\times}_{p,q,r}\cup\C^{(1)\times}_{p,q,r})(\Lambd^{(0)}_r\oplus\C^{n}_{p,q,r})^{\times}\Lambd^{\times}_r,&&\mbox{\!\!\!\!\!\!$n$ is odd},
\\
\P^{\pm\Lambd}_{p,q,r}=(\C^{(0)\times}_{p,q,r}\cup\C^{(1)\times}_{p,q,r})\Lambd_r^{\times},&&\mbox{\!\!\!\!\!\!$n$ is even},
\end{array}
\right.\label{P'}
\\
\!&=&\!\P^{\pm\Lambd}_{p,q,r}\S^{\times}_{p,q,r}
\\
&&\!\!\!\!\!\!\!\!\!\!\!\!\!\quad=
\left\lbrace
\begin{array}{lll}\label{P_Lambd_0}
(\C^{(0)\times}_{p,q,r}\cup\C^{(1)\times}_{p,q,r})(\C^{0}\oplus\C^{n}_{p,q,r})^{\times}\Lambd^{\times}_r,&&\mbox{\!\!$n$ is odd},
\\
(\C^{(0)\times}_{p,q,r}\cup\C^{(1)\times}_{p,q,r})\Lambd_r^{\times},&&\mbox{\!\!$n$ is even},
\end{array}
\right.
\\
&&\!\!\!\!\!\!\!\!\!\!\!\!\!\quad =
\left\lbrace
\begin{array}{lll}\label{P_Lambd_0_}
(\C^{(0)\times}_{p,q,r}\cup\C^{(1)\times}_{p,q,r})(\Lambd_r\oplus\C^{n}_{p,q,r})^{\times},&&\mbox{$n$ is odd and $r \neq n$},
\\
(\C^{(0)\times}_{p,q,r}\cup\C^{(1)\times}_{p,q,r})\Lambd_r^{\times},&&\mbox{$n$ is even or $r=n$},
\end{array}
\right.
\end{eqnarray}
where we get (\ref{P_Lambd_0}) and (\ref{P_Lambd_0_}) by Lemma \ref{rem_eq_P'_easy}. 
Note that in (\ref{P_Lambd})--(\ref{P_Lambd_0_}), $\Lambda^{\times}_r=\ker{(\tilde{\ad})}$ and $\Z^{\times}_{p,q,r}=\ker{(\ad)}$  by Lemma \ref{lemma_3}.

\begin{lem}\label{rem_eq_P'_easy}
We have
\begin{eqnarray}
\!\!\!\!\!\!\!\!\!\!&(\Lambd^{(0)}_r\oplus\C^{n}_{p,q,r})^{\times}\Lambd^{\times}_r=(\C^{0}\oplus\C^{n}_{p,q,r})^{\times}\Lambd^{\times}_r=(\Lambd_{r}\oplus\C^{n}_{p,q,r})^{\times},\qquad r\neq n,\label{rrr}
\\
\!\!\!\!\!\!\!\!\!\!&(\C^{0}\oplus\Lambda^{n}_{n})^{\times}\Lambd^{\times}_n=\Lambd_{n}^{\times},\qquad r=n.\label{rrr_n}
\end{eqnarray}
\end{lem}
\begin{proof}
In the case $r=0$, the proof of the equalities (\ref{rrr}) is trivial, since $\Lambd_0=\Lambd^{(0)}_0=\C^{0}$. Consider the case $r\neq0$. 
By multiplying the factors in the first and the second sets in (\ref{rrr}), we get that each of these sets is a subset of the third one in (\ref{rrr}). Similarly, the left set in (\ref{rrr_n}) is a subset of the right one.
 
Let us show that the third set in (\ref{rrr}) is a subset of the first two ones and that the right set in (\ref{rrr_n}) is a subset of the left one. 
Suppose $T=\alpha e +X+\beta e_{1\ldots n}$, where $\alpha,\beta\in\F$ and $X\in\Lambd_r\setminus(\C^{0}\oplus\C^{n}_{p,q,r})$. So, $T\in(\Lambd_{r}\oplus\C^{n}_{p,q,r})^{\times}$ if $r\neq n$ and $T\in\Lambd_{n}^{\times}$ if $r=n$. Since $T$ is invertible, $\alpha\neq0$ by Lemma \ref{lemma_times_rad}. Then $T=(\alpha e +\beta e_{1\ldots n})(e+\frac{1}{\alpha}X)\in(\C^{0}\oplus\C^{n}_{p,q,r})^{\times}\Lambd^{\times}_{r}$, where the second factor is invertible by Lemma \ref{lemma_times_rad}, and the proof of (\ref{rrr_n}) is completed. In the case $r\neq n$, we have $(\C^{0}\oplus\C^{n}_{p,q,r})^{\times}\Lambd^{\times}_{r}\subseteq(\Lambd^{(0)}_r\oplus\C^{n}_{p,q,r})^{\times}\Lambd^{\times}_r$, and this completes the proof of (\ref{rrr}).
\end{proof}

Let us consider the group $\P^{\pm\rad}_{p,q,r}$:
\begin{eqnarray}
\!\!\P^{\pm\rad}_{p,q,r}:=\P^{\pm}_{p,q,r}(\C^{0}\oplus\rad\C_{p,q,r})^{\times}=(\C^{(0)\times}_{p,q,r}\cup\C^{(1)\times}_{p,q,r})(\C^{0}\oplus\rad\C_{p,q,r})^{\times}.\label{chP_rad}
\end{eqnarray}

\begin{rem}
The groups $\P^{\pm}_{p,q,r}$, $\P_{p,q,r}$, $\P^{\pm\Lambd}_{p,q,r}$, $\P^{\Lambd}_{p,q,r}$, and $\P^{\pm\rad}_{p,q,r}$ are related as follows:
\begin{eqnarray}
&&\!\!\!\!\!\!\!\!\!\!\P_{p,q,r}=\P^{\pm}_{p,q,r}\ker{(\ad)}=\P^{\pm}_{p,q,r}\Z^{\times}_{p,q,r}=\P^{\pm}_{p,q,r}\S^{\times}_{p,q,r},
\\
&&\!\!\!\!\!\!\!\!\!\!\P^{\pm\Lambd}_{p,q,r}=\P^{\pm}_{p,q,r}\ker(\tilde{\ad})=\P^{\pm}_{p,q,r}\Lambd^{\times}_r,
\\
&&\!\!\!\!\!\!\!\!\!\!\P^{\Lambd}_{p,q,r}=\P^{\pm}_{p,q,r}\ker{(\ad)}\ker(\tilde{\ad})=\P^{\pm}_{p,q,r}\Z^{\times}_{p,q,r}\Lambd^{\times}_r=\P^{\pm}_{p,q,r}\S^{\times}_{p,q,r}\Lambd^{\times}_r,
\\
&&\!\!\!\!\!\!\!\!\!\!\P^{\pm\rad}_{p,q,r}=\P^{\pm}_{p,q,r}(\C^{0}\oplus\rad\C_{p,q,r})^{\times},
\end{eqnarray}
$\P^{\pm}_{p,q,r}$ is a subgroup of the groups $\P_{p,q,r}$, $\P^{\Lambd}_{p,q,r}$, $\P^{\pm\Lambd}_{p,q,r}$, $\P^{\pm\rad}_{p,q,r}$;
the groups $\P^{\pm}_{p,q,r}$, $\P_{p,q,r}$, $\P^{\pm\Lambd}_{p,q,r}$ are subgroups of $\P^{\Lambd}_{p,q,r}$; the group $\P^{\pm\Lambda}$ is a subgroup of the group $\P^{\pm\rad}_{p,q,r}$.
\end{rem}
\begin{rem}
In the particular case of the algebra $\C_{p,q,0}$, the groups $\P^{\pm\Lambd}_{p,q,r}$, $\P^{\pm\rad}_{p,q,r}$, and $\P^{\Lambd}_{p,q,r}$ coincide with the groups  $\P^{\pm}$ and $\P$ respectively:
\begin{eqnarray}\label{ppp2}
\P^{\pm\Lambd}_{p,q,0}=\P^{\pm\rad}_{p,q,0}=\P^{\pm}_{p,q,0}=\P^{\pm}\subseteq\P^{\Lambd}_{p,q,0}=\P_{p,q,0}=\P,
\end{eqnarray}
moreover, if $n=p+q$ is even, all the considered groups coincide.
\end{rem}

\begin{rem}\label{rem_00n_ps}
In the case of the Grassmann algebra $\C_{0,0,n}=\Lambd_n$, we have
\begin{eqnarray}
&&\!\!\!\!\!\!\!\!\!\!\!\!\!\!\!\!\!\!\!\!\!\!\!\!\!\P^{\pm}_{0,0,1}\cong\F^{\times}\subset\P_{0,0,1}={\P}^{\pm\Lambd}_{0,0,1}=\P^{\Lambd}_{0,0,1}=\P^{\pm\rad}_{0,0,1}=\Lambd^{\times}_1;\label{PPPP_00n}
\\
&&\!\!\!\!\!\!\!\!\!\!\!\!\!\!\!\!\!\!\!\!\!\!\!\!\!\P^{\pm}_{0,0,n}=\P_{0,0,n}=\Lambd^{(0)\times}_n\subset\P^{\pm\Lambd}_{0,0,n}=\P^{\Lambd}_{0,0,n}=\P^{\pm\rad}_{0,0,n}=\Lambd^{\times}_n,\quad\mbox{$n$ is even};\label{PPPP_00n_1}
\\
&&\!\!\!\!\!\!\!\!\!\!\!\!\!\!\!\!\!\!\!\!\!\!\!\!\!\P^{\pm}_{0,0,n}=\Lambd^{(0)\times}_n\subset\P_{0,0,n}=(\Lambd^{(0)}_n\oplus\Lambd^{n}_n)^{\times}
\\
&&\!\!\!\!\!\!\!\!\!\!\!\!\!\!\!\!\!\!\!\!\!\!\!\!\!\qquad\subset\P^{\pm\Lambd}_{0,0,n}=\P^{\Lambd}_{0,0,n}=\P^{\pm\rad}_{0,0,n}=\Lambd^{\times}_n,\quad\mbox{$n\geq3$ is odd}.\label{PPPP_00n_2}
\end{eqnarray}
The statements $\P^{\pm}_{0,0,n}=\Lambd^{(0)\times}_n$ and $\P_{0,0,n}=\Lambd^{(0)\times}_n(\C^{0}\oplus\Lambd^{n}_{n})^{\times}$ follow from Lemma~\ref{lemma_times_rad}, since any invertible element of $\Lambd_n^{\times}$ has the non-zero projection onto the subspace of grade $0$ and, consequently, is not odd.
\end{rem}

In Theorems \ref{eq_ppmrad}--\ref{eq_P_l1}, we give the equivalent definitions of the groups $\P^{\pm\rad}_{p,q,r}$, $\P^{\Lambd}_{p,q,r}$, $\P^{\pm\Lambd}_{p,q,r}$, $\P_{p,q,r}$, and $\P^{\pm}_{p,q,r}$. We use these definitions to prove Theorems \ref{theorem_g(1)}, \ref{gogn_pqr}, and \ref{gon_pqr}.

\begin{thm}\label{eq_ppmrad}
We have the following equivalent definitions of the group $\P^{\pm\rad}_{p,q,r}$:
\begin{eqnarray}
\P^{\pm\rad}_{p,q,r}&=&(\C^{(0)\times}_{p,q,r}\cup\C^{(1)\times}_{p,q,r})(\C^{0}\oplus\rad\C_{p,q,r})^{\times}\label{eq_chP_rad0}
\\
&=&
\{T\in\C^{\times}_{p,q,r}:\quad \widehat{T^{-1}}T\in(\C^{0}\oplus\rad\C_{p,q,r})^{\times}\}.\label{eq_chP_rad1}
\end{eqnarray}
\end{thm}

\begin{proof}
Let us prove that the set (\ref{eq_chP_rad0}) is a subset of the set (\ref{eq_chP_rad1}). Suppose $T=A B\in{\P}^{\pm\rad}_{p,q,r}$, where $A\in\C^{(0)\times}_{p,q,r}\cup\C^{(1)\times}_{p,q,r}$, $B\in(\C^{0}\oplus\rad\C_{p,q,r})^{\times}$. Then $\widehat{T^{-1}}T=\widehat{(A B)^{-1}}(A B)=\widehat{B^{-1}}\widehat{A^{-1}}A B=\pm\widehat{B^{-1}} A^{-1} A B=\pm\widehat{B^{-1}}B\in(\C^{0}\oplus\rad\C_{p,q,r})^{\times}$, and the proof is completed. 

Now let us prove that the set (\ref{eq_chP_rad1}) is a subset of the set (\ref{eq_chP_rad0}). This statement is proved in the particular case of $\C_{p,q,0}$ in the paper \cite{GenSpin} (see Theorem 1). Consider the case $r\neq 0$. Suppose $T\in\C^{\times}_{p,q,r}$ satisfies $\widehat{T^{-1}}T=W_0+W_1+\beta e_{1\ldots n}\in(\C^{0}\oplus\rad\C_{p,q,r})^{\times}$, where $W_0\in\C^{0}\oplus\rad\C^{(0)}_{p,q,r}\setminus\C^{n}_{p,q,r}$, $W_1\in\rad\C^{(1)}_{p,q,r}\setminus\C^{n}_{p,q,r}$, and $\beta\in\F$. Then $T=\widehat{T}(W_0+W_1+\beta e_{1\ldots n})$. Suppose $T=T_0+T_1$, where $T_0\in\C^{(0)}_{p,q,r}$, $T_1\in\C^{(1)}_{p,q,r}$. 
Then 
\begin{eqnarray}\label{L_eq1_}
T_0+T_1=(T_0-T_1)(W_0+W_1+\beta e_{1\ldots n}).
\end{eqnarray}

Consider the case of even $n$. From the equation (\ref{L_eq1_}) it follows that $T_0=T_0 W_0 +\beta T_0 e_{1\ldots n} - T_1 W_1$, $T_1=-T_1 W_0-\beta T_1 e_{1\ldots n}+T_0 W_1$, i.e. $T_0(e-W_0 -\beta e_{1\ldots n})=-T_1 W_1$, $T_1(e+W_0+\beta e_{1\ldots n})=T_0 W_1$. 
Note that at least one of the elements $e-W_0 -\beta e_{1\ldots n}$ and $e+W_0+\beta e_{1\ldots n}$ has the non-zero projection  onto the subspace of  grade $0$, since otherwise we can sum the equations $\langle e-W_0 -\beta e_{1\ldots n}\rangle_0=0$, $\langle e+W_0+\beta e_{1\ldots n}\rangle_0=0$ and get $\langle 2 e\rangle_0=0$, i.e. a contradiction, where we use the linearity of the projection operator. Then we obtain that at least one of the elements $e-W_0 -\beta e_{1\ldots n}$ and $e+W_0+\beta e_{1\ldots n}$ is invertible by Lemma \ref{lemma_times_rad}. 
Hence, we have at least one of the following two equations:
\begin{eqnarray}
T_0&=&-T_1 W_1 (e-W_0 -\beta e_{1\ldots n})^{-1},\label{eq_1_even}
\\
T_1&=&T_0 W_1 (e+W_0+\beta e_{1\ldots n})^{-1}.\label{eq_2_even}
\end{eqnarray}
Therefore, 
$T=T_0+T_1=T_1(e-W_1 (e-W_0 -\beta e_{1\ldots n})^{-1})\in\C^{(1)\times}_{p,q,r}(\C^{0}\oplus\rad\C_{p,q,r})^{\times}$ or $T=T_0(e+W_1 (e+W_0+\beta e_{1\ldots n})^{-1})\in\C^{(0)\times}_{p,q,r}(\C^{0}\oplus\rad\C_{p,q,r})^{\times}$, where we use that $(e-W_0 -\beta e_{1\ldots n})^{-1}, (e+W_0+\beta e_{1\ldots n})^{-1}\in(\C^{0}\oplus\rad\C_{p,q,r})^{\times}$. In both cases, the second factor in the factorization of $T$ is invertible by Lemma 
 \ref{lemma_times_rad}, since its projection onto the subspace of grade $0$ is non-zero. 
Thus, $T\in(\C^{(0)\times}_{p,q,r}\cup\C^{(1)\times}_{p,q,r})(\C^{0}\oplus\rad\C_{p,q,r})^{\times}$, and the proof is completed.

Consider the case of odd $n$. 
From the equation (\ref{L_eq1_}) it follows that $T_0=T_0 W_0 - T_1 W_1 - \beta T_1 e_{1\ldots n}$, $T_1=T_0 W_1 +\beta T_0 e_{1\ldots n}-T_1 W_0$, i.e. $T_0(e-W_0)=-T_1(W_1+\beta e_{1\ldots n})$, $T_1(e+W_0)=T_0(W_1 +\beta e_{1\ldots n})$.
Since at least one of the elements $e+W_0$ and $e-W_0$ has the non-zero projection onto the subspace of grade $0$, at least one of them is invertible by Lemma \ref{lemma_times_rad}. Therefore, we obtain at least one of the following two equations:
\begin{eqnarray}
T_0&=&-T_1(W_1+\beta e_{1\ldots n})(e-W_0)^{-1},\label{eq_L_2}
\\
T_1&=&T_0(W_1 +\beta e_{1\ldots n})(e+W_0)^{-1}.\label{eq_L_2_}
\end{eqnarray}
Therefore, $T=T_0+T_1=T_1(e-(W_1+\beta e_{1\ldots n})(e-W_0)^{-1})\in\C^{(1)\times}_{p,q,r}(\C^{0}\oplus\rad\C_{p,q,r})^{\times}$ or $T=T_0(e+(W_1 +\beta e_{1\ldots n})(e+W_0)^{-1})\in\C^{(0)\times}_{p,q,r}(\C^{0}\oplus\rad\C_{p,q,r})^{\times}$, where we use that $(e-W_0)^{-1}, (e+W_0)^{-1}\in(\C^{0}\oplus\rad\C_{p,q,r})^{\times}$. 
In both cases, the second factor in the factorization of $T$ is invertible by Lemma \ref{lemma_times_rad}, since its projection onto onto the subspace of grade $0$ is non-zero. 
Thus, $T\in(\C^{(0)\times}_{p,q,r}\cup\C^{(1)\times}_{p,q,r})(\C^{0}\oplus\rad\C_{p,q,r})^{\times}=\P^{\pm\rad}_{p,q,r}$, and the proof is completed.
\end{proof}

\begin{thm}\label{eq_P_l2}
We have the following equivalent definitions of the group $\P^{\Lambd}_{p,q,r}$:
\begin{eqnarray}
\!\!\!\!\!\!\!\!\P^{\Lambd}_{p,q,r}
\!\!\!\!\!&=&\!\!\!\!\!
\left\lbrace
\begin{array}{lll}
(\C^{(0)\times}_{p,q,r}\cup\C^{(1)\times}_{p,q,r})(\C^{0}\oplus\C^{n}_{p,q,r})^{\times}\Lambd_r^{\times},&\quad&\mbox{$n$ is odd},
\\
(\C^{(0)\times}_{p,q,r}\cup\C^{(1)\times}_{p,q,r})\Lambd_r^{\times},&\quad&\mbox{$n$ is even},
\end{array}
\right.\label{eq_p'_0}
\\
\!\!\!\!\!&=&\!\!\!\!\!
\left\lbrace
\begin{array}{lll}
\!\!\{T\in\C^{\times}_{p,q,r}:\;\; \widehat{T^{-1}}T\in(\Lambd_r\oplus\C^{n}_{p,q,r})^{\times}\},\!\!\!\!\!\!\!\!\!&&\mbox{$n$ is odd and $r\neq n$},
\\
\!\!\{T\in\C^{\times}_{p,q,r}:\;\; \widehat{T^{-1}}T\in\Lambd_r^{\times}\},\!\!\!\!\!\!\!\!\!&&\mbox{$n$ is even or $r=n$},
\end{array}
\right.
\label{eq_p'}
\\
\!\!\!\!\!&=&\!\!\!\!\!
\left\lbrace
\begin{array}{lll}
\{T\in\C^{\times}_{p,q,r}:\quad \widehat{T^{-1}}T\in(\Lambd_r\oplus\C^{n}_{p,q,r})^{\times}\},&& r\neq n,
\\
\Lambda^{\times}_{n},&& r=n,
\end{array}
\right.\label{eq_p'_2}
\end{eqnarray}
and the group $\P^{\pm\Lambd}_{p,q,r}$:
\begin{eqnarray}
\!\!\!\!\P^{\pm\Lambd}_{p,q,r} 
\!\!\!\!\!&=&\!\!\!\!\!
(\C^{(0)\times}_{p,q,r}\cup\C^{(1)\times}_{p,q,r})\Lambd_r^{\times}\label{eq_PL0}
\\
\!\!\!\!\!&=&\!\!\!\!\!
\{T\in\C^{\times}_{p,q,r}:\quad \widehat{T^{-1}}T\in\Lambd_r^{\times}\}\label{eq_PL}
\\
\!\!\!\!\!&=&\!\!\!\!\!
\left\lbrace
\begin{array}{lll}
\!\!\{T\in\C^{\times}_{p,q,r}:\;\; \widehat{T^{-1}}T\in\Lambd^{\times}_r\},\!\!\!\!\!\!\!\!\!&&\mbox{$n$ is odd or $r=n$},
\\
\!\!\{T\in\C^{\times}_{p,q,r}:\;\; \widehat{T^{-1}}T\in(\Lambd_r\oplus\C^{n}_{p,q,r})^{\times}\},\!\!\!\!\!\!\!\!\!&&\mbox{$n$ is even and $r\neq n$},
\end{array}\label{eq_PL2}
\right.
\end{eqnarray}
where $\Lambd_r^{\times}=\ker{(\tilde{\ad})}$.
\end{thm}

\begin{proof}
First let us prove (\ref{eq_p'_0})--(\ref{eq_p'_2}).
Let us prove that the set (\ref{eq_p'_0}) is a subset of the set (\ref{eq_p'}).
Suppose $T = A B\in\P^{\Lambd}_{p,q,r}=\P_{p,q,r}\Lambd^{\times}_r$ (\ref{pLambd=pLambd}), where $A\in\P_{p,q,r}$, $B\in\Lambd^{\times}_r$. Then $\widehat{T^{-1}}T=\widehat{(A B)^{-1}}(A B)=\widehat{B^{-1}}\widehat{A^{-1}}A B$. Since $\widehat{A^{-1}}A\in\S^{\times}_{p,q,r}\subseteq\Z^{\times}_{p,q,r}$ by Theorem \ref{eq_P_l1} and since $\widehat{B^{-1}}B\in\Lambd^{\times}_r$, we obtain $\widehat{T^{-1}}T\in\S^{\times}_{p,q,r}\Lambd^{\times}_r$. Therefore, we get $\widehat{T^{-1}}T\in\Lambd_r^{\times}$ in the case of even $n$ or $r=n$, $\widehat{T^{-1}}T\in(\Lambd_r\oplus\C^{n}_{p,q,r})^{\times}$ in the case of odd $n$, $r\neq n$, by Lemma \ref{rem_eq_P'_easy}, and the proof is completed.
It is trivial that the set (\ref{eq_p'}) is a subset of the set (\ref{eq_p'_2}).

Let us prove that the set (\ref{eq_p'_2}) is a subset of the set (\ref{eq_p'_0}). This statement is proved in the particular case $\C_{p,q,0}$ in the paper \cite{OnInner} (see the proof of Theorem 3.2). Consider the case $r\neq0$. Suppose $T\in\C^{\times}_{p,q,r}$ satisfies $\widehat{T^{-1}}T=W_0 + W_1+\beta e_{1\ldots n}$, where $W_0\in\Lambd^{(0)}_r$, $W_1\in\Lambd^{(1)}_r$, and $\beta\in\F$. 
So, $\widehat{T^{-1}}T\in(\Lambd_r\oplus\C^{n}_{p,q,r})^{\times}$ if $r\neq n$ and $\widehat{T^{-1}}T\in\Lambd_n^{\times}$ if $r=n$.
Suppose $T=T_0+T_1$, where $T_0\in\C^{(0)}_{p,q,r}$, $T_1\in\C^{(1)}_{p,q,r}$. Then we obtain the equation  (\ref{L_eq1_}).
Consider the case of even $n$. From the equation (\ref{L_eq1_}) it follows that we have at least one of the equations (\ref{eq_1_even})--(\ref{eq_2_even}) by the proof of Theorem \ref{eq_ppmrad}. Therefore, 
$T=T_0+T_1=T_1(e-W_1 (e-W_0 -\beta e_{1\ldots n})^{-1})\in\C^{(1)\times}_{p,q,r}\Lambd^{\times}_r$ or $T=T_0(e+W_1 (e+W_0+\beta e_{1\ldots n})^{-1})\in\C^{(0)\times}_{p,q,r}\Lambd^{\times}_r$, where we use that $(e-W_0 -\beta e_{1\ldots n})^{-1}, (e+W_0+\beta e_{1\ldots n})^{-1}\in\C^{0}\oplus\Lambd_r\oplus\C^{n}_{p,q,r}$ if $r\neq n$, $(e-W_0 -\beta e_{1\ldots n})^{-1}, (e+W_0+\beta e_{1\ldots n})^{-1}\in\C^{0}\oplus\Lambd_r$ if $r=n$, and $W_1 e_{1\ldots n}=0$. In both cases, the second factor in the factorization of $T$ is invertible by Lemma 
 \ref{lemma_times_rad}, since its projection onto onto the subspace of grade $0$ is non-zero. 
Thus, $T\in(\C^{(0)\times}_{p,q,r}\cup\C^{(1)\times}_{p,q,r})\Lambd^{\times}_r$, and the proof is completed.
Consider the case of odd $n$. 
From the equation (\ref{L_eq1_}) it follows that we have at least one of the equations (\ref{eq_L_2})--(\ref{eq_L_2_}) by the proof of Theorem \ref{eq_ppmrad}.
Therefore, if $r\neq n$, we have $T=T_0+T_1=T_1(e-(W_1+\beta e_{1\ldots n})(e-W_0)^{-1})\in\C^{(1)\times}_{p,q,r}(\Lambd_r\oplus\C^{n}_{p,q,r})^{\times}$  or $T=T_0(e+(W_1 +\beta e_{1\ldots n})(e+W_0)^{-1})\in\C^{(0)\times}_{p,q,r}(\Lambd_r\oplus\C^{n}_{p,q,r})^{\times}$, where we use that $(e-W_0)^{-1}, (e+W_0)^{-1}\in\Lambd_r$. 
In both cases, the second factor in the factorization of $T$ is invertible by Lemma \ref{lemma_times_rad}, since its projection onto onto the subspace of grade $0$ is non-zero. 
Thus, $T\in(\C^{(0)\times}_{p,q,r}\cup\C^{(1)\times}_{p,q,r})(\Lambd_r\oplus\C^{n}_{p,q,r})^{\times}=(\C^{(0)\times}_{p,q,r}\cup\C^{(1)\times}_{p,q,r})(\C^{0}\oplus\C^{n}_{p,q,r})^{\times}\Lambd^{\times}_r$ (see Lemma \ref{rem_eq_P'_easy}), and the proof is completed.
If $r=n$, we have $T=T_0(e+(W_1 +\beta e_{1\ldots n})(e+W_0)^{-1})\in\Lambda^{(0)\times}_{n}\Lambd_n^{\times}=\Lambda^{(0)\times}_{n}(\C^0\oplus\Lambda^{n}_n)^{\times}\Lambd_n^{\times}$ (see Lemma \ref{rem_eq_P'_easy}), and this completes the proof.

Now let us prove (\ref{eq_PL0})--(\ref{eq_PL2}). First let us prove that the set (\ref{eq_PL0}) is a subset of the set (\ref{eq_PL}). Suppose $T=A B\in\check{\P}^{\Lambd}_{p,q,r}$, where $A\in\C^{(0)\times}_{p,q,r}\cup\C^{(1)\times}_{p,q,r}$, $B\in\Lambd^{\times}_r$. Then $\widehat{T^{-1}}T=\widehat{(A B)^{-1}}(A B)=\widehat{B^{-1}}\widehat{A^{-1}}A B=\pm\widehat{B^{-1}} A^{-1} A B=\pm\widehat{B^{-1}}B\in\Lambd^{\times}_r$, and the proof is completed.
It is trivial that the set (\ref{eq_PL}) is a subset of the set (\ref{eq_PL2}).

Let us prove that the set (\ref{eq_PL2}) is a subset of the set (\ref{eq_PL0}). In the case of even $n$, we have proved $\{T\in\C^{\times}_{p,q,r}:\quad \widehat{T^{-1}}T\in(\Lambd_r\oplus\C^{n}_{p,q,r})^{\times}\}=(\C^{(0)\times}_{p,q,r}\cup\C^{(1)\times}_{p,q,r})\Lambd^{\times}_r$ if $r\neq n$ (see (\ref{eq_p'_0}) and (\ref{eq_p'_2})).
If $r=n$, we have $\{T\in\Lambd^{\times}_{n}:\quad \widehat{T^{-1}}T\in\Lambd_n^{\times}\}=\Lambda^{\times}_n\subseteq \Lambda^{(0)\times}_n\Lambda^{\times}_n$.
Consider the case of odd $n$. Suppose $\widehat{T^{-1}}T=W+\beta e_{1\ldots n}\in\Lambd^{\times}_r$, where $\beta=0$, $W\in\Lambd^{\times}_r$. As shown above, we obtain at least one of the equations (\ref{eq_L_2})--(\ref{eq_L_2_}). Hence, we get $T_0=-T_1 W_1(e-W_0)^{-1}$ or $T_1=T_0 W_1(e+W_0)^{-1}$. Therefore, we obtain $T=T_0+T_1=T_1(e-W_1(e-W_0)^{-1})\in\C^{(1)\times}_{p,q,r}\Lambd^{\times}_r$ or $T=T_0(e+W_1(e+W_0)^{-1})\in\C^{(0)\times}_{p,q,r}\Lambd^{\times}_r$, where we use that $(e-W_0)^{-1},(e+W_0)^{-1}\in\Lambd_r$. In both cases, the second factor in the factorization of $T$ is invertible by Lemma \ref{lemma_times_rad}. Thus, $T\in(\C^{(0)\times}_{p,q,r}\cup\C^{(1)\times}_{p,q,r})\Lambd^{\times}_r$, and the proof is completed.
\end{proof}

Note that $\C^{0}\subseteq\Lambd^{(0)}_r\subseteq\C^{0}\oplus\rad\C^{(0)}_{p,q,r}$ in the case of arbitrary $n$ and $\C^{0}\oplus\rad\C^{(0)}_{p,q,r}\subseteq\C^{0n}_{p,q,r}\oplus\rad\C^{(0)}_{p,q,r}$ in the case of odd $n$ or $r=0$ in (\ref{eq_P_})--(\ref{eq_P4}) and (\ref{eq_chP_})--(\ref{eq_chP_3}).

\begin{thm}\label{eq_P_l1}
We have the following equivalent definitions of the group $\P_{p,q,r}$:
\begin{eqnarray}
\!\!\!\!\!\!\!\!\P_{p,q,r}\!\!\!&=&\!\!
\left\lbrace
\begin{array}{lll}\label{P__}
(\C^{(0)\times}_{p,q,r}\cup\C^{(1)\times}_{p,q,r})\C^{0n\times}_{p,q,r},&\;&\mbox{$n$ is odd},
\\
\C^{(0)\times}_{p,q,r}\cup\C^{(1)\times}_{p,q,r},&\;&\mbox{$n$ is even},
\end{array}
\right.
\\
\!\!&=&\!\!\{T\in\C^{\times}_{p,q,r}:\quad \widehat{T^{-1}}T\in\S^{\times}_{p,q,r}\}
\label{eq_P_}
\\
\!\!&=&\!\!\{T\in\C^{\times}_{p,q,r}:\quad \widehat{T^{-1}}T\in\ker(\ad)\}
\label{eq_P}
\\
\!\!&=&\!\!\left\lbrace
\begin{array}{lll}\label{eq_P3}
\!\!\!\{T\in\C^{\times}_{p,q,r}:\;\; \widehat{T^{-1}}T\in(\C^{0n}_{p,q,r}\oplus\rad\C_{p,q,r}^{(0)})^{\times}\},\!\!\!\!\!\!\!&&\mbox{$n$ is odd},
\\
\!\!\!\{T\in\C^{\times}_{p,q,r}:\;\; \widehat{T^{-1}}T\in(\C^{0}\oplus\rad\C_{p,q,r}^{(0)})^{\times}\},\!\!\!\!\!\!\!&&\mbox{$n$ is even},
\end{array}
\right.
\\
\!\!\!\!\!&=&\!\!\!\!\!
\left\lbrace
\begin{array}{lll}\label{eq_P4}
\!\!\!\{T\!\in\!\C^{\times}_{p,q,r}\!:\widehat{T^{-1}}T\!\in\!(\C^{0}\oplus\rad\C^{(0)}_{p,q,r})^{\times}\},\mbox{$n$ is even and $r\neq 0$},
\\
\!\!\!\{T\!\in\!\C^{\times}_{p,q,r}\!:\widehat{T^{-1}}T\!\in\!(\C^{0n}_{p,q,r}\oplus\rad\C^{(0)}_{p,q,r})^{\times}\},\mbox{in the other cases},
\end{array}
\right.
\end{eqnarray}
where in $($\ref{eq_P_}$)$ and $($\ref{eq_P}$)$,
\begin{eqnarray}
\S^{\times}_{p,q,r}&=&
\left\lbrace
\begin{array}{lll}
(\C^{0}\oplus\C^{n}_{p,q,r})^{\times},&&\mbox{$n$ is odd},
\\
\C^{0\times},&&\mbox{$n$ is even},
\end{array}
\right.
\\
\ker(\ad)&=&
\left\lbrace
\begin{array}{lll}
(\Lambd^{(0)}_r\oplus\C^{n}_{p,q,r})^{\times},&&\mbox{$n$ is odd},
\\
\Lambd^{(0)\times}_r,&&\mbox{$n$ is even}.
\end{array}
\right.
\end{eqnarray}
We have the following equivalent definitions of the group $\P^{\pm}_{p,q,r}$:
\begin{eqnarray}
\!\!\!\!\!\!\!\!\!\!\!\P^{\pm}_{p,q,r}\!\!\!\!&=&\!\!\!\C^{(0)\times}_{p,q,r}\cup\C^{(1)\times}_{p,q,r}\label{eq_chP_00}
\\
\!\!\!\!\!&=&\!\!\!\{T\in\C^{\times}_{p,q,r}:\quad \widehat{T^{-1}}T\in\C^{0\times}\}\label{eq_chP_}
\\
\!\!\!\!\!&=&\!\!\!\{T\in\C^{\times}_{p,q,r}:\quad \widehat{T^{-1}}T\in\Lambd^{(0)\times}_r=\ker(\check{\ad})\}\label{eq_chP}
\\
\!\!\!\!\!&=&\!\!\!\{T\in\C^{\times}_{p,q,r}:\quad \widehat{T^{-1}}T\in(\C^{0}\oplus\rad\C_{p,q,r}^{(0)})^{\times}\}\label{eq_chP_3}
\\
\!\!\!\!\!\!\!&=&\!\!\!\!\!
\left\lbrace
\begin{array}{lll}\label{eq_chP_4}
\!\!\!\!\{T\!\in\!\C^{\times}_{p,q,r}\!:\!\widehat{T^{-1}}T\!\in\!(\C^{0n}_{p,q,r}\!\oplus\!\rad\C^{(0)}_{p,q,r})^{\times}\!\},\mbox{$n$ is even and $r=0$},
\\
\!\!\!\!\{T\!\in\!\C^{\times}_{p,q,r}\!:\! \widehat{T^{-1}}T\!\in\!(\C^{0}\oplus\rad\C^{(0)}_{p,q,r})^{\times}\},\mbox{in the other cases}.
\end{array}
\right.
\end{eqnarray}
\end{thm}
\begin{proof}
First let us prove (\ref{P__})--(\ref{eq_P3}). Let us prove that the set (\ref{P__}) is a subset of the set (\ref{eq_P_}). Suppose $T=A B\in\P_{p,q,r}$, where $A\in\C^{(0)\times}_{p,q,r}\cup\C^{(1)\times}_{p,q,r}$ and $B\in(\C^{0}\oplus\C^{n}_{p,q,r})^{\times}$ in the case of odd $n$, $B=e$ in the case of even $n$. Then $\widehat{T^{-1}}T=(\widehat{AB})^{-1}(AB)=\widehat{B^{-1}}\widehat{A^{-1}}A B=\pm \widehat{B^{-1}} A^{-1} A B=\pm \widehat{B^{-1}} B$. We have $\widehat{B^{-1}} B\in(\C^{0}\oplus\C^{n}_{p,q,r})^{\times}$ in the case of odd $n$ (see Remark \ref{inv_gr0n}) and $\widehat{B^{-1}} B=e\in\C^{0\times}$ in the case of even $n$, and the proof is completed. The set (\ref{eq_P_}) is a subset of the set (\ref{eq_P}), which is a subset of the sets (\ref{eq_P3}) and (\ref{eq_P4}), since $\C^{0}\subseteq\Lambd^{(0)}_r\subseteq\C^{0}\oplus\rad\C_{p,q,r}^{(0)}$ in the case of any $n,r$ and $\C^{0}\oplus\rad\C_{p,q,r}^{(0)}\subseteq\C^{0}\oplus\rad\C_{p,q,r}^{(0)}\oplus\C^{n}_{p,q,r}$ if $r=0$ or $n$ is odd. 

Let us prove that the set (\ref{eq_P4}) is a subset of the set (\ref{P__}). This statement is proved in the particular case $\C_{p,q,0}$ in the paper \cite{OnInner} (Theorem 3.2). Consider the case $r\neq0$. Suppose $T\in\C^{\times}_{p,q,r}$ satisfies $\widehat{T^{-1}}T=W_0+W_1+\beta e_{1\ldots n}$, where $W_0\in\C^{0}\oplus\rad\C^{(0)}_{p,q,r}$, $W_1=0$, $\beta\in\F$ in the case of arbitrary $n$, and $\beta=0$ if $n$ is even. Suppose $T=T_0+T_1$, where $T_0\in\C^{(0)}_{p,q,r}$, $T_1\in\C^{(1)}_{p,q,r}$. Then we obtain the equation (\ref{L_eq1_}).
Consider the case of even $n$. From the equation (\ref{L_eq1_}) it follows that we obtain one of the equations (\ref{eq_1_even})--(\ref{eq_2_even}) by the proof of Theorem \ref{eq_ppmrad}. Therefore, we get either $T_0=0$ or $T_1=0$, since $W_1=0$. Thus, $T\in\C^{(0)\times}_{p,q,r}\cup\C^{(1)\times}_{p,q,r}=\P^{\pm}_{p,q,r}$, and the proof is completed.
Consider the case of odd $n$. From the equation (\ref{L_eq1_}) it follows that we have one of the equations (\ref{eq_L_2})--(\ref{eq_L_2_}) by the proof of Theorem \ref{eq_ppmrad}.
Therefore, either $T=T_0+T_1=T_1(e-\beta e_{1\ldots n}(e-W_0)^{-1})\in\C^{(1)\times}_{p,q,r}(\C^{0}\oplus\C^{n}_{p,q,r})^{\times}$ or $T=T_0(e+\beta e_{1\ldots n}(e+W_0)^{-1})\in\C^{(0)\times}_{p,q,r}(\C^{0}\oplus\C^{n}_{p,q,r})^{\times}$, where in both cases, the second factor is invertible by Lemma \ref{lemma_times_rad}.
Thus, $T\in(\C^{(0)\times}_{p,q,r}\cup\C^{(1)\times}_{p,q,r})\Z_{p,q,r}^{\times}$.

Now let us prove (\ref{eq_chP_00})--(\ref{eq_chP_3}).
The set (\ref{eq_chP_00}) is a subset of the set (\ref{eq_chP_}), since we obtain $\widehat{T^{-1}}T=\pm T^{-1}T=\pm e\in\C^{0\times}$ for any $T\in\C^{(0)\times}_{p,q,r}\cup\C^{(1)\times}_{p,q,r}$. The set (\ref{eq_chP_}) is a subset of the set (\ref{eq_chP}), which is a subset of the sets (\ref{eq_chP_3}) and (\ref{eq_chP_4}), since $\C^{0}\subseteq\Lambd^{(0)}_r\subseteq\C^{0}\oplus\rad\C_{p,q,r}^{(0)}$ in the case of arbitrary $n,r$ and $\C^{0}\oplus\rad\C_{p,q,r}^{(0)}\subseteq\C^{0}\oplus\rad\C^{(0)}_{p,q,r}\oplus\C^{n}_{p,q,r}$ if $r=0$ or $n$ is odd.

Let us prove that the set (\ref{eq_chP_4}) is a subset of the set (\ref{eq_chP_00}). In the case of even $n$, we have proved $\{T\in\C^{\times}_{p,q,r}:\quad \widehat{T^{-1}}T\in(\C^{0}\oplus\rad\C_{p,q,r}^{(0)}\oplus\C^{n}_{p,q,r})^{\times}\}=\C^{(0)\times}_{p,q,r}\cup\C^{(1)\times}_{p,q,r}$ if $r=0$ and $\{T\in\C^{\times}_{p,q,r}:\quad \widehat{T^{-1}}T\in(\C^{0}\oplus\rad\C_{p,q,r}^{(0)})^{\times}\}=\C^{(0)\times}_{p,q,r}\cup\C^{(1)\times}_{p,q,r}$  if $r\neq 0$ (see. (\ref{eq_P4}) and (\ref{P__})). Consider the case of odd $n$. Suppose $\widehat{T^{-1}}T=W_0+\beta e_{1\ldots n}\in\C^{0}\oplus\rad\C^{(0)}_{p,q,r}$, where $\beta=0$, $W_0\in\C^{0}\oplus\rad\C^{(0)}_{p,q,r}$.
As shown above, we obtain one of the equations  (\ref{eq_L_2})--(\ref{eq_L_2_}).
Since $\beta=0$, we get either $T_0=0$ or $T_1=0$; thus, $T\in\C^{(0)\times}_{p,q,r}\cup\C^{(1)\times}_{p,q,r}$ and the proof is completed.
\end{proof}

\section{Examples on the groups $\P^{\pm}_{p,q,r}$, $\P_{p,q,r}$, $\P^{\pm\Lambd}_{p,q,r}$, $\P^{\Lambd}_{p,q,r}$, and $\P^{\pm\rad}_{p,q,r}$}\label{section_examples}

Let us give some examples on the groups $\P^{\pm}_{p,q,r}$, $\P_{p,q,r}$, $\P^{\pm\Lambd}_{p,q,r}$, $\P^{\Lambd}_{p,q,r}$ and $\P^{\pm\rad}_{p,q,r}$ in the cases of the low-dimensional degenerate geometric algebras $\C_{p,q,r}$. We use that the degenerate geometric algebra can be embedded into the non-degenerate geometric algebra of larger dimension (see Clifford -- Jordan -- Wigner representation \cite{CJW2}), which is isomorphic to the matrix algebra (see, for example, \cite{lounesto,p}).

Let us consider the groups of upper triangular matrices $\UT(2,\F)$ and $\UT(4,\F)$ (see, for example, \cite{baker}):
\begin{eqnarray}
\UT(2,\F)&:=&\{\begin{bmatrix}
x_{11} & x_{12} \\
0 & x_{21} \\
\end{bmatrix}\in\GL(2,\F)\},\label{ut_2}
\\
\UT(4,\F)&:=&\{\begin{bmatrix}
x_{11} & x_{12} & x_{13} & x_{14}\\
0 & x_{22} & x_{23} & x_{24} \\
0 & 0 & x_{33} & x_{34} \\
0 & 0 & 0 & x_{44}\\
\end{bmatrix}\in\GL(4,\F)\},\label{ut_4}
\end{eqnarray}
and a unipotent subgroup $\SUT(2,\F)$ \cite{baker} of the group $\UT(2,\F)$:
\begin{eqnarray}\label{sut}
\SUT(2,\F)&:=&\{\begin{bmatrix}
1 & x_{12} \\
0 & 1 \\
\end{bmatrix},\; x_{12}\in\F\}.
\end{eqnarray}

\begin{ex}
Consider the degenerate algebra $\Lambd_1=\C_{0,0,1}$, which can be embedded into the non-degenerate algebra $\C_{1,1,0}\cong \Mat(2,\F)$. 
For the elements $e$ and $e_1$, we have
\begin{eqnarray*}
e\mapsto
\begin{bmatrix}
1 & 0 \\
0 & 1 \\
\end{bmatrix},\qquad
e_1\mapsto
\begin{bmatrix} 
0 & 1 \\
0 & 0 \\
\end{bmatrix}.
\end{eqnarray*}
We obtain $\P^{\pm}_{0,0,1}=\Lambd^{(0)\times}_1=\C^{0\times}\cong\F^{\times}$ and
\begin{eqnarray}
\P_{0,0,1}&=&\P^{\pm\Lambd}_{0,0,1}=\P^{\Lambd}_{0,0,1}=\P^{\pm\rad}_{0,0,1}=\Lambd^{\times}_{1}
\\
&\cong&
\{
\begin{bmatrix} 
x_0  & x_1 \\
0 & x_0 \\
\end{bmatrix}:\; x_0,x_1\in\F,\; x_0\neq0
\}.\label{matrix_1}
\end{eqnarray}
Let us note that all the introduced groups can be realized as subgroups of the group of upper triangular matrices $\UT(2,\F)$ (\ref{ut_2}), which is a Borel subgroup (see, for example, \cite{baker}) of the general linear group $\GL(2,\F)$. Also note that the unitriangular group $\SUT(2,\F)$ (\ref{sut}) is a subgroup of the considered matrix group (\ref{matrix_1}).
\end{ex}

Let us consider the higher-dimensional Heisenberg group $\Heis_4$ (see, for example, \cite{baker,hall}):
\begin{eqnarray}\label{heis}
\Heis_4:=
\{\begin{bmatrix}
1 & x_{12} & x_{13} & x_{14}\\
0 & 1 & 0 & x_{24} \\
0 & 0 & 1 & x_{34} \\
0 & 0 & 0 & 1\\
\end{bmatrix}\in\GL(4,\F)\}.
\end{eqnarray}

\begin{ex}
Since $\Lambd_2=\C_{0,0,2}$ can be embedded into the algebra $\C_{2,2,0}\cong \Mat(4,\F)$, we get
\begin{eqnarray}
\P^{\pm\Lambd}_{0,0,2}&=&\P^{\Lambd}_{0,0,2}=\P^{\pm\rad}_{0,0,2}=\Lambd^{\times}_2
\\
&\cong&\{
\begin{bmatrix} 
x_0 & x_1 & x_2 & x_3 \\
0 & x_0 & 0 & - x_2 \\
0 & 0 & x_0 & x_1 \\
0 & 0& 0 & x_0 \\
\end{bmatrix}:\;\; x_0,x_1,x_2,x_3\in\F,\;\; x_0\neq0
\},
\end{eqnarray}
and
\begin{eqnarray}
\P^{\pm}_{0,0,2}&=&\P_{0,0,2}=\Lambd^{(0)\times}_2
\\
&\cong&
\{
\begin{bmatrix}
x_0 & 0 & 0 & x_3  \\
0& x_0 & 0 & 0 \\
0 & 0 & x_0 & 0 \\
0 & 0 & 0 & x_0 \\
\end{bmatrix}:\;\; x_0,x_3\in\F,\;\; x_0\neq0
\}.
\end{eqnarray}
These matrix groups are subgroups of the group of upper triangular matrices $\UT(4,\F)$ (\ref{ut_4}). Also note that all the introduced Lie groups are closely related to the higher-dimensional Heisenberg group $\Heis_4$ (\ref{heis}).
\end{ex}

\begin{ex}
Let us consider the degenerate algebra $\C_{1,0,1}$. It can be embedded into $\C_{2,1,0}\cong\Mat(2,\F)\oplus\Mat(2,\F)$. We obtain
\begin{eqnarray}
\P^{\pm}_{1,0,1}&\cong&
\{
\begin{bmatrix}
x_0 & x_3 & 0 & 0 \\
0 & x_0 & 0 & 0 \\
0 & 0 & x_0 & x_3 \\
0 & 0 & 0 & x_0 \\
\end{bmatrix}\in\GL{(4,\F)}
\\
&&
\cup
\begin{bmatrix}
x_1 & x_2 & 0 & 0 \\
0 & -x_1 & 0 & 0 \\
0 & 0 & -x_1 & -x_2 \\
0 & 0 & 0 & x_1 \\
\end{bmatrix}\in\GL{(4,\F)}
\}.
\end{eqnarray}
Note that this matrix group is a subgroup of $\UT(4,\F)$ (\ref{ut_4}). Also note that this group is closely related to $\Heis_4$ (\ref{heis}) as well as the groups in the previous example.
\end{ex}

\section{The groups preserving the subspaces of fixed parity under 
the adjoint and twisted adjoint representations
}\label{section_gamma}

We use the following notation for the groups preserving the subspaces of fixed parity under $\ad$ (\ref{ar}):
\begin{eqnarray}\label{g(k)_notation}
{\Gamma}_{p,q,r}^{(k)}:=\{T\in\C^{\times}_{p,q,r}:\quad \ad_T(\C^{(k)}_{p,q,r})=T\C^{(k)}_{p,q,r}T^{-1}\subseteq\C^{(k)}_{p,q,r}\},\quad k=0,1,
\end{eqnarray}
under $\check{\ad}$ (\ref{twa1}):
\begin{eqnarray}\label{g(k)_notation_}
\check{\Gamma}_{p,q,r}^{(k)}:=\{T\in\C^{\times}_{p,q,r}:\quad \check{\ad}_T(\C^{(k)}_{p,q,r})=\widehat{T}\C^{(k)}_{p,q,r}T^{-1}\subseteq\C^{(k)}_{p,q,r}\},\quad k=0,1,
\end{eqnarray}
and under $\tilde{\ad}$ (\ref{twa2}):
\begin{eqnarray}\label{g(k)_notation__}
    \tilde{\Gamma}_{p,q,r}^{(k)}:=\{T\in\C^{\times}_{p,q,r}:\quad \tilde{\ad}_T(\C^{(k)}_{p,q,r})\subseteq\C^{(k)}_{p,q,r}\},\quad k=0,1.
\end{eqnarray}

\begin{thm}\label{theorem_g(1)}
We have
\begin{eqnarray}
&&\P_{p,q,r}={\Gamma}_{p,q,r}^{(1)}\subseteq\P^{\Lambd}_{p,q,r}={\Gamma}_{p,q,r}^{(0)}=\tilde{\Gamma}_{p,q,r}^{(0)},\label{theorem_g(1)_0}
\\
&&\P^{\pm}_{p,q,r}=\check{\Gamma}_{p,q,r}^{(0)}\subseteq\P^{\pm\Lambd}_{p,q,r}=\check{\Gamma}_{p,q,r}^{(1)}=\tilde{\Gamma}_{p,q,r}^{(1)}.\label{theorem_g(1)_1}
\end{eqnarray}
\end{thm}
\begin{proof}
The statements $\P_{p,q,r}\subseteq\P^{\Lambd}_{p,q,r}$ and $\P^{\pm}_{p,q,r}\subseteq\P^{\pm\Lambd}_{p,q,r}$ follow from the definitions of the groups (\ref{P+-}), (\ref{P__0_}), (\ref{P_Lambd}), and (\ref{P_Lambd_0}). We obtain $\tilde{\Gamma}_{p,q,r}^{(0)}={\Gamma}_{p,q,r}^{(0)}$ and $\tilde{\Gamma}_{p,q,r}^{(1)}=\check{\Gamma}_{p,q,r}^{(1)}$, since $\tilde{\ad}_T(\C^{(0)}_{p,q,r})=\ad_T(\C^{(0)}_{p,q,r})$ and $\tilde{\ad}_T(\C^{(1)}_{p,q,r})=\check{\ad}_T(\C^{(1)}_{p,q,r})$ by (\ref{ad_t_1}) and (\ref{ad_t_2}) respectively. 

Let us prove $\P^{\pm}_{p,q,r}\subseteq\check{\Gamma}_{p,q,r}^{(0)}$. Suppose $T\in\P^{\pm}_{p,q,r}$ (\ref{P+-}). If $T\in\C^{(0)\times}_{p,q,r}$, then $\widehat{T}=T$ and $T^{-1}\in\C^{(0)\times}_{p,q,r}$. If $T\in\C^{(1)\times}_{p,q,r}$, then $\widehat{T}=-T$ and $T^{-1}\in\C^{(1)\times}_{p,q,r}$. In both cases, we obtain $\widehat{T}\C^{(0)}_{p,q,r}T^{-1}\subseteq\C^{(0)}_{p,q,r}$ by (\ref{even}). Thus, $T\in\check{\Gamma}_{p,q,r}^{(0)}$.
Let us prove $\P_{p,q,r}\subseteq{\Gamma}_{p,q,r}^{(1)}$. Suppose $T=X W\in\P_{p,q,r}=\P^{\pm}_{p,q,r}\Z^{\times}_{p,q,r}$ (\ref{P_}), where $X\in\C^{(0)\times}_{p,q,r}\cup\C^{(1)\times}_{p,q,r}$ and $W\in\Z^{\times}_{p,q,r}$. Then we get $T\C^{(1)}_{p,q,r}T^{-1}= X W \C^{(1)}_{p,q,r} W^{-1} X^{-1}=X  \C^{(1)}_{p,q,r} W W^{-1} X^{-1}=X \C^{(1)}_{p,q,r}X^{-1}\subseteq\C^{(1)}_{p,q,r}$, where we use (\ref{even}). Thus, $T\in{\Gamma}_{p,q,r}^{(1)}$.

Let us prove $\P^{\Lambd}_{p,q,r}\subseteq{\Gamma}_{p,q,r}^{(0)}$. Suppose $T=X W\in\P^{\Lambd}_{p,q,r}$ (\ref{eq_p'_0}), where $X\in\C^{(0)\times}_{p,q,r}\cup\C^{(1)\times}_{p,q,r}$, $W\in(\Lambd_r\oplus\C^{n}_{p,q,r})^{\times}$ in the case of odd $n$ and $r\neq n$ and $W\in\Lambd_r^{\times}$  in the case of even $n$ or $r=n$. We obtain $T\C^{(0)}_{p,q,r}T^{-1}= X W \C^{(0)}_{p,q,r} W^{-1} X^{-1}=X \C^{(0)}_{p,q,r} W W^{-1} X^{-1}= X \C^{(0)}_{p,q,r} X^{-1}\subseteq\C^{(0)}_{p,q,r}$, where we use the property (\ref{even}) and that $W \C^{(0)}_{p,q,r}=\C^{(0)}_{p,q,r} W$ by Lemma \ref{lemma_XVVX}. Thus, $T\in{\Gamma}_{p,q,r}^{(0)}$.
Let us prove $\P^{\pm\Lambd}_{p,q,r}\subseteq\check{\Gamma}_{p,q,r}^{(1)}$. Suppose $T=X W\in\P^{\pm\Lambd}_{p,q,r}$ (\ref{P_Lambd}), where $X\in\C^{(0)\times}_{p,q,r}\cup\C^{(1)\times}_{p,q,r}$ and $W\in\Lambd^{\times}_r$. Since $\widehat{W} e_a = e_a W$ for any generator $e_a$, $a=1,\ldots, n$, by Lemma \ref{lemma_XVVX_gr1} and since any odd basis element can be represented as a product of an odd number of generators, we get $\widehat{W}\C^{(1)}_{p,q,r}=\C^{(1)}_{p,q,r}W$ by Lemma \ref{modd^XU=UX}. Then we obtain $\widehat{T}\C^{(1)}_{p,q,r}T^{-1}=\widehat{X}\widehat{W}\C^{(1)}_{p,q,r}W^{-1} X^{-1}=\pm X \C^{(1)}_{p,q,r} W W^{-1} X^{-1}=\pm X \C^{(1)}_{p,q,r} X^{-1}\subseteq\C^{(1)}_{p,q,r}$ by (\ref{even}). Thus, $T\in\check{\Gamma}_{p,q,r}^{(1)}$.

Let us prove ${\Gamma}_{p,q,r}^{(1)}\subseteq\P_{p,q,r}$. Suppose $T\in\C^{\times}_{p,q,r}$ satisfies $T \C^{(1)}_{p,q,r} T^{-1}\subseteq\C^{(1)}_{p,q,r}$; then we obtain
$T U T^{-1}=-(T U T^{-1})\widehat{\;\;}=\widehat{T} U \widehat{T^{-1}}$ for any $U\in\C^{(1)}_{p,q,r}$.
Multiplying both sides of this equation on the left by $\widehat{T^{-1}}$, on the right by $T$, we get
\begin{eqnarray}\label{ttuutt1_r}
(\widehat{T^{-1}}T)U=U(\widehat{T^{-1}}T),\qquad \forall U\in\C^{(1)}_{p,q,r}.
\end{eqnarray}
In particular, (\ref{ttuutt1_r}) is true for any generator $U=e_a\in\C^{(1)}_{p,q,r}$, $a=1,\ldots, n$. Since the identity element $U=e\in\C^{0}$ satisfies (\ref{ttuutt1_r}) as well, we obtain $\ad_{\widehat{T^{-1}}T}(U)=U$ for any $U\in\C_{p,q,r}$.
Therefore, $\widehat{T^{-1}}T\in\ker{(\ad)}$. Thus, $T\in\P_{p,q,r}$ by Theorem \ref{eq_P_l1}.

Let us prove $\check{\Gamma}^{(1)}\subseteq\P^{\pm\Lambd}_{p,q,r}$. Suppose $T\in\C^{\times}_{p,q,r}$ satisfies $\widehat{T}\C^{(1)}_{p,q,r} T^{-1}\subseteq\C^{(1)}_{p,q,r}$. Then we get $\widehat{T} U T^{-1} = - (\widehat{T} U T^{-1})\widehat{\;\;} = T U \widehat{T^{-1}}$ for any $U\in\C^{(1)}_{p,q,r}$.
Multiplying both sides of the equation on the left by $T^{-1}$, on the right by $T$, we obtain $T^{-1} \widehat{T} U = U\widehat{T^{-1}} T$, i.e. $\widehat{(\widehat{T^{-1}}T)} U = U(\widehat{T^{-1}}T)$ for any $U\in\C^{(1)}_{p,q,r}$. In particular, this equation is true for any generator $U=e_a\in\C^{1}_{p,q,r}$, $a=1,\ldots,n$. Using Lemma \ref{lemma_XVVX_gr1}, we get $\widehat{T^{-1}}T\in\Lambd_r^{\times}$; hence, $T\in\P^{\pm\Lambd}_{p,q,r}$ by Theorem~\ref{eq_P_l2}.

Let us prove ${\Gamma}_{p,q,r}^{(0)}\subseteq\P^{\Lambd}_{p,q,r}$. Suppose $T\in\C^{\times}_{p,q,r}$ satisfies $T \C^{(0)}_{p,q,r} T^{-1}\subseteq\C^{(0)}_{p,q,r}$. Then we get $T U T^{-1}=(T U T^{-1})\widehat{\;\;}=\widehat{T} U \widehat{T^{-1}}$ for any $U\in\C^{(0)}_{p,q,r}$.
Multiplying both sides of this equation on the left by $\widehat{T^{-1}}$, on the right by $T$, we obtain $(\widehat{T^{-1}}T)U=U(\widehat{T^{-1}}T)$ for any $U\in\C^{(0)}_{p,q,r}$.
Using Lemma \ref{lemma_XVVX}, we have $\widehat{T^{-1}}T\in\Lambd_r\oplus\C^{n}_{p,q,r}$ if $r\neq n$ and $\widehat{T^{-1}}T\in\Lambd_n$, i.e. $T\in\Lambda_n^{\times}$, if $r= n$.
Thus, $T\in\P^{\Lambd}_{p,q,r}$ by Theorem \ref{eq_P_l2}.

Let us prove $\check{\Gamma}_{p,q,r}^{(0)}\subseteq\P^{\pm}_{p,q,r}$. This statement is proved in the case $r=0$ in the paper \cite{GenSpin}. Consider the case $r\neq0$. Suppose $T\in\C^{\times}_{p,q,r}$ satisfies $\widehat{T} \C^{(0)}_{p,q,r} T^{-1}\subseteq\C^{(0)}_{p,q,r}$. Then $\widehat{T} U T^{-1} = (\widehat{T} U T^{-1})\widehat{\;\;} = T U \widehat{T^{-1}}$ for any  $U\in\C^{(0)}_{p,q,r}$.
Multiplying both sides of this equation on the left by $T^{-1}$, on the right by $T$, we obtain $T^{-1} \widehat{T} U = U\widehat{T^{-1}} T$, i.e. $\widehat{(\widehat{T^{-1}}T)} U = U(\widehat{T^{-1}}T)$ for any $U\in\C^{(0)}_{p,q,r}$. 
Using Lemma \ref{lemma_XVVX}, we get $\widehat{T^{-1}}T\in(\Lambd^{(0)}_r\oplus\C^{n}_{p,q,r})^{\times}$ in the case of even $n$, $r\neq n$, and $\widehat{T^{-1}}T\in\Lambd^{(0)\times}_r$ in the case $n$ is odd or $r=n$ is even. 
Therefore, $T\in\P^{\pm}_{p,q,r}$ by (\ref{eq_chP}) if $n$ is odd or $r=n$ is even and by (\ref{eq_chP_3}) if $n$ is even and $r\neq n$, since $\Lambd^{(0)}_r\oplus\C^{n}_{p,q,r}\subseteq\C^{0}\oplus\rad\C^{(0)}_{p,q,r}$.
\end{proof}

\begin{rem}\label{p,q,0_P}
In the particular case $r=0$, we have by (\ref{ppp1}) and (\ref{ppp2}):
\begin{eqnarray}
&&\P^{\pm}_{p,q,0}=\P^{\pm\Lambd}_{p,q,0}=\check{\Gamma}_{p,q,0}^{(0)}=\check{\Gamma}_{p,q,0}^{(1)}=\tilde{\Gamma}_{p,q,0}^{(1)}
\\
&&\quad \subset\P_{p,q,0}=\P^{\Lambd}_{p,q,0}={\Gamma}_{p,q,0}^{(1)}={\Gamma}_{p,q,0}^{(0)}=\tilde{\Gamma}_{p,q,0}^{(0)},\quad\mbox{$n$ is odd},
\end{eqnarray}
and
\begin{eqnarray}
&&\P^{\pm}_{p,q,0}=\P^{\pm\Lambd}_{p,q,0}=\check{\Gamma}_{p,q,0}^{(0)}=\check{\Gamma}_{p,q,0}^{(1)}=\tilde{\Gamma}_{p,q,0}^{(1)}
\\
&&\quad =\P_{p,q,0}=\P^{\Lambd}_{p,q,0}={\Gamma}_{p,q,0}^{(1)}={\Gamma}_{p,q,0}^{(0)}=\tilde{\Gamma}_{p,q,0}^{(0)},\quad\mbox{$n$ is even}.
\end{eqnarray}
\end{rem}

\begin{rem}
In the particular case of the Grassmann algebra $\C_{0,0,n}=\Lambd_n$, we have three different groups:
\begin{eqnarray}
&&\!\!\!\!\!\!\!\!\!\!\!\!\!\!\!\!\!\!\P^{\pm}_{0,0,n}=\check{\Gamma}_{0,0,n}^{(0)}=\ker{(\check{\ad})}=\Lambd^{(0)\times}_n,
\\
&&\!\!\!\!\!\!\!\!\!\!\!\!\!\!\!\!\!\!\P^{\Lambd}_{0,0,n}=\P^{\pm\Lambd}_{0,0,n}={\Gamma}_{0,0,n}^{(0)}=\tilde{\Gamma}_{0,0,n}^{(0)}=\check{\Gamma}_{0,0,n}^{(1)}=\tilde{\Gamma}_{0,0,n}^{(1)}=\ker(\tilde{\ad})=\Lambd^{\times}_n,
\\
&&\!\!\!\!\!\!\!\!\!\!\!\!\!\!\!\!\!\!\P_{0,0,n}={\Gamma}_{0,0,n}^{(1)}=\ker{(\ad)}=
\left\lbrace
\begin{array}{lll}
(\Lambd^{(0)}_n\oplus\Lambd_n^n)^{\times}&\quad&\mbox{if $n$ is odd},
\\
\Lambd^{(0)\times}_n&\quad&\mbox{if $n$ is even}.
\end{array}
\right.
\end{eqnarray}
\end{rem}

\section{The groups ${\Gamma}_{p,q,r}^{0}$, ${\Gamma}_{p,q,r}^n$, ${\Gamma}_{p,q,r}^{0n}$, $\check{\Gamma}_{p,q,r}^{0}$, $\check{\Gamma}_{p,q,r}^n$, $\check{\Gamma}_{p,q,r}^{0n}$, $\tilde{\Gamma}_{p,q,r}^{0}$, $\tilde{\Gamma}_{p,q,r}^n$, and $\tilde{\Gamma}_{p,q,r}^{0n}$ }\label{section_gamma0n}

Let us use the following notation for the groups preserving the subspace of the fixed grade
 $k$ under $\ad$ (\ref{ar}):
\begin{eqnarray}\label{g_k_1}
{\Gamma}_{p,q,r}^{k}:=\{T\in\C^{\times}_{p,q,r}:\quad \ad_T(\C^{k}_{p,q,r})=T\C^{k}_{p,q,r}T^{-1}\subseteq\C^{k}_{p,q,r}\},
\end{eqnarray}
under $\check{\ad}$ (\ref{twa1}):
\begin{eqnarray}\label{g_k_2}
\check{\Gamma}_{p,q,r}^{k}:=\{T\in\C^{\times}_{p,q,r}:\quad \check{\ad}_T(\C^{k}_{p,q,r})=\widehat{T}\C^{k}_{p,q,r}T^{-1}\subseteq\C^{k}_{p,q,r}\},
\end{eqnarray}
and under $\tilde{\ad}$ (\ref{twa2}):
\begin{eqnarray}\label{g_k_3}
\tilde{\Gamma}_{p,q,r}^{k}:=\{T\in\C^{\times}_{p,q,r}:\quad \tilde{\ad}_T(\C^{k}_{p,q,r})\subseteq\C^{k}_{p,q,r}\}.
\end{eqnarray}
The groups $\tilde{\Gamma}_{p,q,r}^k$ are related with the groups ${\Gamma}_{p,q,r}^k$ and $\check{\Gamma}_{p,q,r}^k$ in the following way:
\begin{eqnarray}\label{def_tilde_gk}
    \tilde{\Gamma}_{p,q,r}^k=
    \left\lbrace
    \begin{array}{lll}
    \check{\Gamma}_{p,q,r}^k,&&\mbox{$k$ is odd},
    \\
    {\Gamma}_{p,q,r}^k,&&\mbox{$k$ is even},
    \end{array}
    \right.
\end{eqnarray}
since $\tilde{\ad}_T(\C^{k}_{p,q,r})={\ad}_T(\C^{k}_{p,q,r})$ in the case of even $k$ by (\ref{ad_t_1}) and $\tilde{\ad}_T(\C^{k}_{p,q,r})=\check{\ad}_T(\C^{k}_{p,q,r})$ in the case of odd $k$ by (\ref{ad_t_2}).
In this section, we consider only the groups ${\Gamma}_{p,q,r}^{0}$, ${\Gamma}_{p,q,r}^n$, $\check{\Gamma}_{p,q,r}^0$, $\check{\Gamma}_{p,q,r}^n$, $\tilde{\Gamma}_{p,q,r}^0$, and $\tilde{\Gamma}_{p,q,r}^n$ (the cases of $k=0,n$), 
since these groups are related with the groups $\P^{\pm}_{p,q,r}$ and $\P^{\pm\rad}_{p,q,r}$  discussed in Sections \ref{section_P}--\ref{section_gamma} above. The groups ${\Gamma}_{p,q,r}^k$, $\check{\Gamma}_{p,q,r}^k$, and $\tilde{\Gamma}_{p,q,r}^k$, $k=1,\ldots,n-1$, differ significantly from the introduced groups even in the particular case of the non-degenerate algebra $\C_{p,q,0}$ \cite{GenSpin,OnInner}.

\begin{thm}\label{gogn_pqr}
We have
\begin{eqnarray}
\!\!\!\!\!\!\!\!\!\!\!\!\!\!\!&&{\Gamma}_{p,q,r}^{0}=\tilde{\Gamma}_{p,q,r}^0=\C^{\times}_{p,q,r}, \qquad {\Gamma}_{p,q,r}^{n}=
\left\lbrace
\begin{array}{lll}
\C^{\times}_{p,q,r},&&\mbox{$n$ is odd},
\\
\tilde{\Gamma}_{p,q,r}^n=\P^{\pm\rad}_{p,q,r},&&\mbox{$n$  is even},
\end{array}
\right.
\\
\!\!\!\!\!\!\!\!\!\!\!\!\!\!\!&&\check{\Gamma}_{p,q,r}^{0}=\P^{\pm}_{p,q,r},\qquad \check{\Gamma}_{p,q,r}^{n}=
\left\lbrace
\begin{array}{lll}
\tilde{\Gamma}_{p,q,r}^n=\P^{\pm\rad}_{p,q,r},&&\mbox{$n$ is odd},
\\
\C^{\times}_{p,q,r},&&\mbox{$n$ is even}.
\end{array}
\right.
\end{eqnarray}
\end{thm}
\begin{proof}
We obtain $\tilde{\Gamma}_{p,q,r}^0={\Gamma}_{p,q,r}^{0}$ in the case of arbitrary $n$, $\tilde{\Gamma}_{p,q,r}^n={\Gamma}_{p,q,r}^{n}$ in the case of even $n$, and $\tilde{\Gamma}_{p,q,r}^n=\check{\Gamma}_{p,q,r}^{n}$ in the case of odd $n$, using (\ref{def_tilde_gk}). Now it remains to consider only the groups ${\Gamma}_{p,q,r}^0$, ${\Gamma}_{p,q,r}^n$, $\check{\Gamma}_{p,q,r}^0$, and $\check{\Gamma}_{p,q,r}^n$.

We have ${\Gamma}_{p,q,r}^{0}=\C^{\times}_{p,q,r}$ in the case of arbitrary $n$, since $T \C^{0}T^{-1}\subseteq\C^{0}$ is true for any $T\in\C^{\times}_{p,q,r}$. We obtain ${\Gamma}_{p,q,r}^{n}=\C^{\times}_{p,q,r}$ if $n$ is odd, since $T e_{1\ldots n} T^{-1}=e_{1\ldots n} T T^{-1}=e_{1\ldots n}\in\C^{n}_{p,q,r}$ for any $T\in\C^{\times}_{p,q,r}$ by $e_{1\ldots n}\in\Z_{p,q,r}$. We get $\check{\Gamma}_{p,q,r}^n=\C^{\times}_{p,q,r}$ if $n$ is even, since $\widehat{T}e_{1\ldots n}T^{-1}=e_{1\ldots n} T T^{-1}=e_{1\ldots n}\in\C^{n}_{p,q,r}$ for any $T\in\C^{\times}_{p,q,r}$, since $e_{1\ldots n}$ commutes with all even elements and anticommutes with all odd elements. 

Let us prove $\P^{\pm}_{p,q,r}\subseteq\check{\Gamma}_{p,q,r}^0$. Suppose $T\in\P^{\pm}_{p,q,r}=\C^{(0)\times}_{p,q,r}\cup\C^{(1)\times}_{p,q,r}$; then $\widehat{T}=\pm T$ and $\widehat{T}\C^{0}T^{-1}=\pm T \C^{0} T^{-1}\subseteq\C^{0}$, and the proof is completed. 
Let us prove $\check{\Gamma}_{p,q,r}^0\subseteq\P^{\pm}_{p,q,r}$. Suppose $T\in\C^{\times}_{p,q,r}$ satisfies $\widehat{T}\C^{0}T^{-1}\subseteq\C^{0}$; then $\widehat{T}T^{-1}=\alpha e$, where $\alpha\in\F^{\times}$, i.e.
$\widehat{T}=\alpha T$. Suppose $T=T_0+T_1$, where $T_0\in\C^{(0)}_{p,q,r}$ and $T_1\in\C^{(1)}_{p,q,r}$; then we get $T_0-T_1=\alpha T_0+\alpha T_1$, 
i.e. $T_0=\alpha T_0$ and $-T_1=\alpha T_1$. If $\alpha=1$, then $T_1=0$ and $T\in\C^{(0)\times}_{p,q,r}$. If $\alpha=-1$, then $T_0=0$ and
$T\in\C^{(1)\times}_{p,q,r}$. If $\alpha\neq 1,-1$, then $T_0=T_1=0$, and we get a contradiction. Thus, $T\in\C^{(0)\times}_{p,q,r}\cup\C^{(1)\times}_{p,q,r}=\P^{\pm}_{p,q,r}$.

Let us prove ${\Gamma}_{p,q,r}^{n}\subseteq\P^{\pm\rad}_{p,q,r}$ in the case of even $n$. Suppose $T\in\C^{\times}_{p,q,r}$ satisfies $T e_{1\ldots n} T^{-1}=\alpha e_{1\ldots n}$, where $\alpha\in\F^{\times}$. 
Multiplying both sides of this equation on the left by $T^{-1}$, on the right by $\frac{1}{\alpha}T$, we get $\frac{1}{\alpha}e_{1\ldots n}=T^{-1} e_{1\ldots n} T$. Then we obtain $e_{1\ldots n}\widehat{T^{-1}}T=\frac{1}{\alpha}e_{1\ldots n}$, where we use that $e_{1\ldots n}$ commutes with all even elements and anticommutes with all odd elements. Therefore, $\widehat{T^{-1}}T\in(\C^{0}\oplus\rad\C_{p,q,r})^{\times}$ and $T\in\P^{\pm\rad}_{p,q,r}$ by Theorem \ref{eq_ppmrad}. Let us prove $\P^{\pm\rad}_{p,q,r}\subseteq{\Gamma}_{p,q,r}^{n}$ if $n$ is even. Suppose $\widehat{T^{-1}}T=\alpha e + W\in(\C^{0}\oplus\rad\C_{p,q,r})^{\times}$, where $\alpha\in\F^{\times}$ and $W\in\rad\C_{p,q,r}$. Multiplying both sides of this equation on the left by  $e_{1\ldots n}$, we obtain $e_{1\ldots n}\widehat{T^{-1}}T=\alpha e_{1\ldots n}$, where we use that $e_{1\ldots n}W=0$. Therefore, $T^{-1} e_{1\ldots n} T =\alpha e_{1\ldots n}$. Multiplying both sides of the equation on the left by $\frac{1}{\alpha}T$, on the right by $T^{-1}$, we get $T e_{1\ldots n} T^{-1}=\frac{1}{\alpha}e_{1\ldots n}\in\C^{n}_{p,q,r}$. Thus, $T\in{\Gamma}_{p,q,r}^{n}$.

Let us prove $\check{\Gamma}_{p,q,r}^{n}\subseteq\P^{\pm\rad}_{p,q,r}$ if $n$ is odd. Suppose $T\in\C^{\times}_{p,q,r}$ satisfies $\widehat{T} e_{1\ldots n}T^{-1}=\alpha e_{1\ldots n}$, where $\alpha\in\F^{\times}$. Multiplying both sides of this equation on the left by $\widehat{T^{-1}}$, on the right by $\frac{1}{\alpha}T$, we get $\frac{1}{\alpha}e_{1\ldots n}=\widehat{T^{-1}}e_{1\ldots n}T$. Then we obtain $e_{1\ldots n}\widehat{T^{-1}}T=\frac{1}{\alpha}e_{1\ldots n}$, since $e_{1\ldots n}\in\Z_{p,q,r}$. Therefore, $\widehat{T^{-1}}T\in(\C^{0}\oplus\rad\C_{p,q,r})^{\times}$ and $T\in\P^{\pm\rad}_{p,q,r}$ by Theorem \ref{eq_ppmrad}. Let us prove that $\P^{\pm\rad}_{p,q,r}\subseteq\check{\Gamma}_{p,q,r}^{n}$ if $n$ is odd. Suppose $\widehat{T^{-1}}T=\alpha e + W\in(\C^{0}\oplus\rad\C_{p,q,r})^{\times}$, where $\alpha\in\F^{\times}$ and $W\in\rad\C_{p,q,r}$. Multiplying both sides of the equation on the left by $e_{1\ldots n}$, we obtain $e_{1\ldots n}\widehat{T^{-1}}T=\alpha e_{1\ldots n}$. Therefore, $\widehat{T^{-1}} e_{1\ldots n} T =\alpha e_{1\ldots n}$. Multiplying both sides of this equation on the left by $\frac{1}{\alpha}\widehat{T}$, on the right by $T^{-1}$, we obtain $\widehat{T} e_{1\ldots n} T^{-1}=\frac{1}{\alpha}e_{1\ldots n}\in\C^{n}_{p,q,r}$. Thus, $T\in\check{\Gamma}_{p,q,r}^{n}$.
\end{proof}

\begin{rem}
In the case of the non-degenerate algebra $\C_{p,q,0}$, we have the following statements, which are proved in the papers \cite{OnInner} and \cite{GenSpin} respectively:
\begin{eqnarray}\label{stat}
{\Gamma}_{p,q,0}^{k}={\Gamma}_{p,q,0}^{n-k},\qquad \check{\Gamma}_{p,q,0}^k=\check{\Gamma}_{p,q,0}^{n-k},\qquad k=1,\ldots,n-1.
\end{eqnarray}
Note that in the case of the degenerate algebra $\C_{p,q,r}$, $r\neq0$, the statements (\ref{stat}) are not true. Let us consider the following example.
In the case $\C_{0,0,3}$, $n=r=3$, consider the element $T=e+e_1$, which is invertible, since $(e+e_1)(e-e_1)=e$. 

We have $T\not\in{\Gamma}_{0,0,3}^{1}$, since $T e_2 T^{-1}=(e+e_1) e_2 (e-e_1)=e_{2}+2e_{12}\not\in\C^{1}_{0,0,3}$.
We obtain $T\in{\Gamma}_{0,0,3}^{2}$, since $T e_{ab}T^{-1}=e_{ab}\in\C^{2}_{0,0,3}$ for $a,b=1,2,3$, $a<b$. Thus, ${\Gamma}_{p,q,r}^{1}\neq{\Gamma}_{p,q,r}^{2}$.

We have $T\in\check{\Gamma}_{0,0,3}^1$, since $\widehat{T}e_1T^{-1}=(e-e_1)e_a(e-e_1)=e_a\in\C^{1}_{0,0,3}$ for $a=1,2,3$. We get $T\not\in\check{\Gamma}_{0,0,3}^2$, since $\widehat{T}e_{23}T^{-1}=(e-e_1)e_{23}(e-e_1)=(e_{23}-e_{123})(e-e_1)=e_{23}-2e_{123}\not\in\C^{2}_{0,0,3}$. Thus, $\check{\Gamma}_{p,q,r}^1\neq \check{\Gamma}_{p,q,r}^2$.
\end{rem}

Let us consider the groups preserving the direct sum of the subspaces $\C^{0}$ and $\C^{n}_{p,q,r}$ under  $\ad$ (\ref{ar}) and  $\check{\ad}$ (\ref{twa1}) respectively:
\begin{eqnarray}
{\Gamma}_{p,q,r}^{0n}&:=&\{T\in\C^{\times}_{p,q,r}:\quad {\ad}(\C^{0n}_{p,q,r})=T\C^{0n}_{p,q,r}T^{-1}\subseteq\C^{0n}_{p,q,r}\},\label{g_0n_1}
\\
\check{\Gamma}_{p,q,r}^{0n}&:=&\{T\in\C^{\times}_{p,q,r}:\quad \check{\ad}(\C^{0n}_{p,q,r})=\widehat{T}\C^{0n}_{p,q,r}T^{-1}\subseteq\C^{0n}_{p,q,r}\}.\label{g_0n_2}
\end{eqnarray}
Also we consider the groups preserving the subspace $\C^{0n}_{p,q,r}$ under $\tilde{\ad}$ (\ref{twa2}):
\begin{eqnarray}
    \tilde{\Gamma}_{p,q,r}^{0n}&:=&\{T\in\C^{\times}_{p,q,r}:\quad\tilde{\ad}(\C^{0n}_{p,q,r})\subseteq\C^{0n}_{p,q,r}\}\label{g_0n_3}
    \\
    &=&
    \left\lbrace
    \begin{array}{lll}\label{deftildeg0n}
    \{T\in\C^{\times}_{p,q,r}:\;\;\widehat{T}\C^{n}_{p,q,r}T^{-1}\subseteq\C^{0n}_{p,q,r}\},\!\!\!&&\mbox{$n$ is odd},
    \\\{T\in\C^{\times}_{p,q,r}:\;\;T\C^{0n}_{p,q,r}T^{-1}\subseteq\C^{0n}_{p,q,r}\},\!\!\!&&\mbox{$n$ is even}.
    \end{array}
    \right.
\end{eqnarray}

\begin{thm}\label{gon_pqr}
We have
\begin{eqnarray}
&&{\Gamma}_{p,q,r}^{0n}={\Gamma}_{p,q,r}^{n}=
\left\lbrace
\begin{array}{lll}\label{g0n_1eq}
\C^{\times}_{p,q,r},&&\mbox{$n$ is odd},
\\
\P^{\pm\rad}_{p,q,r},&&\mbox{$n$ is even},
\end{array}
\right.
\\
&&\check{\Gamma}_{p,q,r}^{0n}=\P_{p,q,r},\label{g0n_2eq}
\\
&&\tilde{\Gamma}_{p,q,r}^{0n}=
\left\lbrace
\begin{array}{lll}
\P_{p,q,0},&&\mbox{$n$ is odd and $r=0$},
\\
\P^{\pm\rad}_{p,q,r},&&\mbox{in the other cases}.
\end{array}
\right.\label{tildeg0n}
\end{eqnarray}
\end{thm}
\begin{proof}
Let us prove (\ref{g0n_1eq}).
In the case of odd $n$, the statement ${\Gamma}_{p,q,r}^{0n}=\C^{\times}_{p,q,r}$ follows from ${\Gamma}_{p,q,r}^{0}={\Gamma}_{p,q,r}^{n}=\C^{\times}_{p,q,r}$ (Lemma \ref{gogn_pqr}).
Consider the case of even $n$. If $r=0$, then we have ${\Gamma}_{p,q,r}^{0n}=\P^{\pm}=\P^{\pm\rad}_{p,q,0}$ by Lemma $2$ \cite{GenSpin}. Consider the case $r\neq0$. Let us prove ${\Gamma}_{p,q,r}^{0n}\subseteq\P^{\pm\rad}_{p,q,r}$. 
Suppose $T\C^{n}_{p,q,r}T^{-1}\subseteq(\C^{0}\oplus\C^{n}_{p,q,r})$; then $T e_{1\ldots n} T^{-1}=e_{1\ldots n}\widehat{T}T^{-1}=\alpha e +\beta e_{1\ldots n}$, $\alpha,\beta\in\F$,
where we use that $e_{1\ldots n}$ commutes with all even elements and anticommutes with all odd elements. Since $\langle e_{1\ldots n}X\rangle_0=0$ for any $X\in\C_{p,q,r}$, we get $\langle e_{1\ldots n}\widehat{T}T^{-1}\rangle_0=0$; hence, $\alpha=0$, i.e. $T\C^{n}_{p,q,r}T^{-1}\subseteq\C^{n}_{p,q,r}$. 
Thus, $T\in{\Gamma}_{p,q,r}^{n}=\P^{\pm\rad}_{p,q,r}$ (Lemma \ref{gogn_pqr}) and the proof is completed. Let us prove $\P^{\pm\rad}_{p,q,r}\subseteq{\Gamma}_{p,q,r}^{0n}$. Since $\P^{\pm\rad}_{p,q,r}={\Gamma}_{p,q,r}^{n}$ and ${\Gamma}_{p,q,r}^{0}=\C^{\times}_{p,q,r}$ by Lemma \ref{gogn_pqr}, we get $\P^{\pm\rad}_{p,q,r}={\Gamma}_{p,q,r}^{n}={\Gamma}_{p,q,r}^{0}\cap{\Gamma}_{p,q,r}^{n}\subseteq{\Gamma}_{p,q,r}^{0n}$, and the proof is completed.

Now let us prove (\ref{g0n_2eq}). First let us prove $\P_{p,q,r}\subseteq\check{\Gamma}_{p,q,r}^{0n}$. 
In the case of even $n$, we have $\P_{p,q,r}=\P^{\pm}_{p,q,r}=\check{\Gamma}_{p,q,r}^{0}=\check{\Gamma}_{p,q,r}^{0}\cap\check{\Gamma}_{p,q,r}^{n}\subseteq\check{\Gamma}_{p,q,r}^{0n}$, where we use Lemma \ref{gogn_pqr}. Consider the case of odd $n$. Suppose $T=X Y\in\P_{p,q,r}$, where $X\in\C^{(0)\times}_{p,q,r}\cup\C^{(1)\times}_{p,q,r}$, $Y\in\Z^{\times}$. For $\alpha,\beta\in\F$, we get $\widehat{T}(\alpha e +\beta e_{1\ldots n})T^{-1}=\widehat{(XY)}(\alpha e +\beta e_{1\ldots n})(XY)^{-1}=\pm X \widehat{Y}(\alpha e +\beta e_{1\ldots n})Y^{-1}X^{-1}$ $=\pm (\alpha e +\beta e_{1\ldots n}) X X^{-1} \widehat{Y}Y^{-1}$ $=\pm (\alpha e +\beta e_{1\ldots n}) \widehat{Y}Y^{-1}\in(\C^{0}\oplus\C^{n}_{p,q,r})$, where we use $\widehat{X}=\pm X$. Thus, $T\in\check{\Gamma}_{p,q,r}^{0n}$.
Let us prove $\check{\Gamma}_{p,q,r}^{0n}\subseteq\P_{p,q,r}$.
Suppose $\widehat{T}\C^{0}T^{-1}\subseteq(\C^{0}\oplus\C^{n}_{p,q,r})$, i.e. $\widehat{T}T^{-1}=\alpha e +\beta e_{1\ldots n}\in(\C^{0}\oplus\C^{n}_{p,q,r})$, where $\alpha,\beta\in\F$. Multiplying both sides of the equation on the left by $\widehat{T^{-1}}$, on the right by $T$, we get $\alpha \widehat{T^{-1}}T +\beta \widehat{T^{-1}}e_{1\ldots n} T=e$.
Consider the case of even $n$. We obtain $\alpha \widehat{T^{-1}}T=e-\beta e_{1\ldots n}\in(\C^{0}\oplus\C^{n}_{p,q,r})^{\times}$; hence, $T\in\P^{\pm}_{p,q,r}=\P_{p,q,r}$ by Lemma \ref{eq_P_l1}. Consider the case of odd $n$. We get $(\alpha e + \beta e_{1\ldots n})\widehat{T^{-1}}T=e$. Since $\widehat{T^{-1}}T\in\C^{\times}_{p,q,r}$ and $e\in\C^{\times}_{p,q,r}$, we have $(\alpha e + \beta e_{1\ldots n})\in\C^{\times}_{p,q,r}$. Then  $\widehat{T^{-1}}T=\alpha e -\beta e_{1\ldots n}\in(\C^{0}\oplus\C^{n}_{p,q,r})^{\times}$; therefore, $T\in\P_{p,q,r}$ by Lemma \ref{eq_P_l1}. 

Finally, let us prove (\ref{tildeg0n}). In the case of even $n$, we obtain $\tilde{\Gamma}^{0n}_{p,q,r}=\Gamma^{0n}_{p,q,r}=\P^{\pm\rad}_{p,q,r}$, using (\ref{deftildeg0n}) and (\ref{g0n_1eq}). Let us consider the case of odd $n$. If $r=0$, then $e_{1\ldots n}$ is invertible; therefore, we obtain $\tilde{\Gamma}^{0n}_{p,q,0}=\{T\in\C^{\times}_{p,q,0}:\;
\C^{n}_{p,q,0}\widehat{T}T^{-1}\subseteq\C^{0n}_{p,q,0}\}=\{T\in\C^{\times}_{p,q,0}:\;\widehat{T}T^{-1}\in\C^{0n}_{p,q,0}\}=\{T\in\C^{\times}_{p,q,0}:\; \widehat{T}\C^{0n}_{p,q,0}T^{-1}\subseteq\C^{0n}_{p,q,0}\}=\check{\Gamma}^{0n}_{p,q,0}=\P_{p,q,0}$, where we use (\ref{g0n_2eq}).
If $r\neq0$, then $\langle e_{1\ldots n} X\rangle_0=0$ for any $X\in\C_{p,q,r}$; therefore, we get $\tilde{\Gamma}^{0n}_{p,q,r}=\{T\in\C^{\times}_{p,q,r}:\;
\C^{n}_{p,q,r}\widehat{T}T^{-1}\subseteq\C^{0n}_{p,q,r}\}=\{T\in\C^{\times}_{p,q,r}:\;
\C^{n}_{p,q,r}\widehat{T}T^{-1}\subseteq\C^{n}_{p,q,r}\}=\{T\in\C^{\times}_{p,q,r}:\;
\widehat{T}\C^{n}_{p,q,r}T^{-1}\subseteq\C^{n}_{p,q,r}\}=\check{\Gamma}^{n}_{p,q,r}=\P^{\pm\rad}_{p,q,r}$, where we use Theorem \ref{gogn_pqr}, and the proof is completed. 
\end{proof}

\begin{rem}\label{rem_go_00n}
In the particular case of the Grassmann algebra $\C_{0,0,n}=\Lambda_n$, we get from Theorems \ref{gogn_pqr} and \ref{gon_pqr}:
\begin{eqnarray}
\!\!\!\!\!\!\!\!\!\!\!\!&&\check{\Gamma}_{0,0,1}^{0}=\C^{0\times}\subset {\Gamma}_{0,0,1}^{0}=\tilde{\Gamma}_{0,0,1}^{0}={\Gamma}_{0,0,1}^{1}=\check{\Gamma}_{0,0,1}^1=\tilde{\Gamma}_{0,0,1}^{1}
\\
\!\!\!\!\!\!\!\!\!\!\!\!&&\qquad ={\Gamma}_{0,0,1}^{01}=\check{\Gamma}_{0,0,1}^{01}=\tilde{\Gamma}_{0,0,1}^{01}=\Lambda^{\times}_1,\quad\mbox{$n=1$};
\\
\!\!\!\!\!\!\!\!\!\!\!\!&&\check{\Gamma}_{0,0,n}^{0}=\check{\Gamma}_{0,0,n}^{0n}=\Lambda^{(0)\times}_n\subset {\Gamma}_{0,0,n}^{0}=\tilde{\Gamma}_{0,0,n}^{0}={\Gamma}_{0,0,n}^{n}=\check{\Gamma}_{0,0,n}^n
\\
\!\!\!\!\!\!\!\!\!\!\!\!&&\qquad=\tilde{\Gamma}_{0,0,n}^{n}={\Gamma}_{0,0,n}^{0n}=\tilde{\Gamma}_{0,0,n}^{0n}=\Lambda^{\times}_n,\quad\mbox{$n$ is even};
\\
\!\!\!\!\!\!\!\!\!\!\!\!&&\check{\Gamma}_{0,0,n}^{0}=\Lambda^{(0)\times}_n\subset\check{\Gamma}_{0,0,n}^{0n}=(\Lambda^{(0)}_r\oplus\Lambda^{n}_{r})^{\times}\subset {\Gamma}_{0,0,n}^{0}=\tilde{\Gamma}_{0,0,n}^{0}={\Gamma}_{0,0,n}^{n}\\
\!\!\!\!\!\!\!\!\!\!\!\!&&\qquad
=\check{\Gamma}_{0,0,n}^n=\tilde{\Gamma}_{0,0,n}^{n}={\Gamma}_{0,0,n}^{0n}=\tilde{\Gamma}_{0,0,n}^{0n}=\Lambda^{\times}_n,\quad\mbox{$n\geq3$ is odd},
\end{eqnarray}
where we use Remark \ref{rem_00n_ps}. 
\end{rem}

\begin{rem}
In the particular case of the non-degenerate geometric algebra $\C_{p,q,0}$, we obtain the statements from the papers \cite{GenSpin} and \cite{OnInner}:
\begin{eqnarray}
\!\!\!\!\!\!\!\!\!\!\!\!\!\!\!\!\!\!\!\!\!\!&&{\Gamma}_{p,q,0}^{0}=\tilde{\Gamma}_{p,q,0}^{0}=\C^{\times}_{p,q,0},\quad{\Gamma}_{p,q,0}^{n}={\Gamma}_{p,q,0}^{0n}=
\left\lbrace
\begin{array}{lll}
\C^{\times}_{p,q,0},\!\!\!\!\!\!\!\!&&\mbox{$n$ is odd},
\\
\P^{\pm},\!\!\!\!\!\!\!\!&&\mbox{$n$ is even},
\end{array}
\right.\label{rem_75_1}
\\
\!\!\!\!\!\!\!\!\!\!\!\!\!\!\!\!\!\!\!\!\!\!&&\check{\Gamma}_{p,q,0}^{0}=\tilde{\Gamma}_{p,q,0}^{n}=\P^{\pm},\;\; \check{\Gamma}_{p,q,0}^{0n}=\tilde{\Gamma}_{p,q,0}^{0n}=\P,\;\;
 \check{\Gamma}_{p,q,0}^{n}=
\left\lbrace
\begin{array}{lll}
\P^{\pm},\!\!\!\!\!\!\!\!&&\mbox{$n$ is odd},
\\
\C^{\times}_{p,q,0},\!\!\!\!\!\!\!\!&&\mbox{$n$ is even},
\end{array}
\right.\label{rem_75_2}
\end{eqnarray}
i.e. we have
\begin{eqnarray}
\!\!\!\!\!\!\!\!&&\check{\Gamma}_{p,q,0}^0=\check{\Gamma}_{p,q,0}^{n}=\tilde{\Gamma}_{p,q,0}^{n}=\P^{\pm}\subset{\Gamma}_{p,q,0}^{0}=\tilde{\Gamma}_{p,q,0}^{0}={\Gamma}_{p,q,0}^{n}
\\
\!\!\!\!\!\!\!\!&&\qquad={\Gamma}_{p,q,0}^{0n}=\check{\Gamma}_{p,q,0}^{0n}=\tilde{\Gamma}_{p,q,0}^{0n}=\P=\C^{\times}_{p,q,0},\quad\mbox{$n=1$},
\\
\!\!\!\!\!\!\!\!&&\check{\Gamma}_{p,q,0}^0={\Gamma}_{p,q,0}^{n}=\tilde{\Gamma}_{p,q,0}^{n}={\Gamma}_{p,q,0}^{0n}=\check{\Gamma}_{p,q,0}^{0n}=\tilde{\Gamma}_{p,q,0}^{0n}=\P^{\pm}=\P
\\
\!\!\!\!\!\!\!\!&&\qquad\subset{\Gamma}_{p,q,0}^{0}=\tilde{\Gamma}_{p,q,0}^{0}=\check{\Gamma}_{p,q,0}^n=\C^{\times}_{p,q,0},\quad\mbox{$n$ is even},
\\
\!\!\!\!\!\!\!\!&&\check{\Gamma}_{p,q,0}^0=\check{\Gamma}_{p,q,0}^{n}=\tilde{\Gamma}_{p,q,0}^{n}=\P^{\pm}\subset\check{\Gamma}_{p,q,0}^{0n}=\tilde{\Gamma}_{p,q,0}^{0n}=\P
\\
\!\!\!\!\!\!\!\!&&\qquad\subset{\Gamma}_{p,q,0}^{0}=\tilde{\Gamma}_{p,q,0}^{0}={\Gamma}_{p,q,0}^{n}={\Gamma}_{p,q,0}^{0n}=\C^{\times}_{p,q,0},\quad\mbox{$n\geq3$ is odd}.
\end{eqnarray}
\end{rem}

\section{The corresponding Lie algebras}\label{lie_alg}

Let us denote the Lie algebras of the Lie groups $\P^{\pm}_{p,q,r}$, $\P_{p,q,r}$, $\P^{\pm\Lambd}_{p,q,r}$, $\P^{\Lambd}_{p,q,r}$, and $\P^{\pm\rad}_{p,q,r}$ by $\mathfrak{p}^{\pm}_{p,q,r}$, $\mathfrak{p}_{p,q,r}$, $\mathfrak{p}^{\pm\Lambd}_{p,q,r}$, $\mathfrak{p}^{\Lambd}_{p,q,r}$, and $\mathfrak{p}^{\pm\rad}_{p,q,r}$ respectively.

\begin{thm}\label{Lie_alg_th}
We have the Lie algebras
\begin{eqnarray}
\mathfrak{p}^{\pm}_{p,q,r}&=&\C^{(0)}_{p,q,r};\label{alg_1}
\\
\mathfrak{p}^{\pm\Lambd}_{p,q,r}&=&\C^{(0)}_{p,q,r}\oplus\Lambd^{(1)}_r;
\\
\mathfrak{p}^{\pm\rad}_{p,q,r}&=&\C^{(0)}_{p,q,r}\oplus\rad\C^{(1)}_{p,q,r};
\\
\mathfrak{p}_{p,q,r}&=&
\left\lbrace
\begin{array}{lll}
\C^{(0)}_{p,q,r}\oplus\C^{n}_{p,q,r},&& \mbox{$n$ is odd};
\\
\C^{(0)}_{p,q,r},&& \mbox{$n$ is even};
\end{array}
\right.
\\
\mathfrak{p}^{\Lambd}_{p,q,r}&=&
\left\lbrace
\begin{array}{lll}\label{alg_5}
\C^{(0)}_{p,q,r}\oplus\Lambd^{(1)}_r\oplus\C^{n}_{p,q,r},&&\mbox{$n$ is odd},\quad n\neq r;
\\
\C^{(0)}_{p,q,r}\oplus\Lambd^{(1)}_r,&&\mbox{in the other cases}
\end{array}
\right.
\end{eqnarray}
of the following dimensions:
\begin{eqnarray*}
\dim\mathfrak{p}^{\pm}_{p,q,r}&=&2^{n-1};
\\
\dim\mathfrak{p}^{\pm\Lambd}_{p,q,r}&=&
\left\lbrace
\begin{array}{lll}
2^{n-1}+2^{r-1},&&\quad r\geq 1;
\\
2^{n-1},&&\quad r=0;
\end{array}
\right.
\\
\dim\mathfrak{p}^{\pm\rad}_{p,q,r}&=&
\left\lbrace
\begin{array}{lll}
2^{n}-2^{p+q-1},&&\quad p+q\geq1;
\\
2^{n},&&\quad p=q=0;
\end{array}
\right.
\\
\dim\mathfrak{p}_{p,q,r}&=&
\left\lbrace
\begin{array}{lll}
2^{n-1}+1,&& \;\mbox{$n$ is odd};
\\
2^{n-1},&& \;\mbox{$n$ is even};
\end{array}
\right.
\\
\dim\mathfrak{p}^{\Lambd}_{p,q,r}&=&
\left\lbrace
\begin{array}{lll}
2^{n-1}+2^{r-1}+1,&&\;\mbox{$n$ is odd},\quad n\neq r,\quad r\geq1;
\\
2^{n-1}+1,&&\;\mbox{$n$ is odd},\quad r=0;
\\
2^{n-1},&&\;\mbox{$n$ is even},\quad r=0;
\\
2^{n-1}+2^{r-1},&&\;\mbox{in the other cases}.
\end{array}
\right.
\end{eqnarray*}
The sets on the right-hand sides of $($\ref{alg_1}$)$--$($\ref{alg_5}$)$ are considered with respect to the commutator $[U,V]=UV-VU$.
\end{thm}
\begin{proof}
We use the well-known facts about the relation between an arbitrary Lie group and the corresponding Lie algebra in
order to prove the statements.
 We calculate the dimensions of the considered Lie algebras using
$\dim{\C^{(0)}_{p,q,r}}=2^{n-1}$, $ \dim{\Lambd^{(1)}_{r}}=2^{r-1}$ if $r\geq1$, $ \dim{\Lambd^{(1)}_{0}}=0$, $ \dim{\C^{n}_{p,q,r}}=1$, $\dim(\rad\C^{(1)}_{0,0,n})=\dim\C^{(1)}_{0,0,n}=2^{n-1}$, $\dim(\rad\C^{(1)}_{p,q,r})=2^{n-1}-2^{p+q-1}$ if $p+q\geq1$.
\end{proof}

\begin{rem}
In the particular case of the non-degenerate algebra $\C_{p,q,0}$, we obtain 
\begin{eqnarray}
&&\mathfrak{p}^{\pm}_{p,q,0}=\mathfrak{p}^{\pm\Lambd}_{p,q,0}=\mathfrak{p}^{\pm\rad}_{p,q,0}=\C^{(0)}_{p,q,0},
\\
&&\mathfrak{p}_{p,q,0}=\mathfrak{p}^{\Lambd}_{p,q,0}=
\left\lbrace
\begin{array}{lll}
\C^{(0)}_{p,q,0}\oplus\C^{n}_{p,q,0},&&\mbox{$n$ is odd};
\\
\C^{(0)}_{p,q,0},&&\mbox{$n$ is even}.
\end{array}
\right.
\end{eqnarray}
\end{rem}

\begin{rem}
In the case of the Grassmann algebra $\C_{0,0,n}=\Lambda_{n}$, we obtain from Theorem \ref{Lie_alg_th}
\begin{eqnarray}
& \mathfrak{p}^{\pm}_{0,0,n}=\Lambda^{(0)}_n,\qquad 
\mathfrak{p}^{\pm\Lambda}_{0,0,n}=\mathfrak{p}^{\pm\rad}_{0,0,n}=\mathfrak{p}^{\Lambda}_{0,0,n}=\Lambda_n,
\\
& \mathfrak{p}_{0,0,n}=
\left\lbrace
\begin{array}{lll}
\Lambda^{(0)}_n\oplus\Lambda^{n}_{n},&& \mbox{$n$ is odd};
\\
\Lambda^{(0)}_{n},&& \mbox{$n$ is even}.
\end{array}
\right.
\end{eqnarray}
\end{rem}

\section{Conclusions}\label{section_conclusions}

In this paper, we introduce and study the five families of Lie groups $\P^{\pm}_{p,q,r}$, $\P_{p,q,r}$,  $\P^{\pm\Lambd}_{p,q,r}$, $\P^{\Lambd}_{p,q,r}$, and $\P^{\pm\rad}_{p,q,r}$ in the real and complex degenerate Clifford geometric algebras $\C_{p,q,r}$ of arbitrary dimension and signature:
\begin{eqnarray}
&\P^{\pm}_{p,q,r}=\C^{(0)\times}_{p,q,r}\cup\C^{(1)\times}_{p,q,r},\quad \P_{p,q,r}=\P^{\pm}_{p,q,r}\Z^{\times}_{p,q,r},\quad
\P^{\pm\Lambd}_{p,q,r}=\P^{\pm}_{p,q,r}\Lambda^{\times}_r,\label{our_groups_1}
\\
&\P^{\Lambd}_{p,q,r}=\P^{\pm}_{p,q,r}\Z^{\times}_{p,q,r}\Lambda^{\times}_r,\quad\P^{\pm\rad}_{p,q,r}=\P^{\pm}_{p,q,r}(\C^{0}\oplus\rad\C_{p,q,r})^{\times}.\label{our_groups_2}
\end{eqnarray}
 These groups preserve several fundamental subspaces under the adjoint representation and the twisted adjoint representation. 
  The groups (\ref{our_groups_1})--(\ref{our_groups_2}) are closely related to the spin groups, the Lipschitz groups, and the Clifford groups in the degenerate case, and that is why they are interesting for consideration. 

We provide several equivalent definitions of the groups $\P^{\pm}_{p,q,r}$, $\P_{p,q,r}$,  $\P^{\pm\Lambd}_{p,q,r}$, $\P^{\Lambd}_{p,q,r}$, and $\P^{\pm\rad}_{p,q,r}$ in Theorems~\ref{eq_ppmrad}--\ref{eq_P_l1}.
We prove that some of these groups preserve the even and odd subspaces under $\ad$ ($\P_{p,q,r}$ $={\Gamma}_{p,q,r}^{(1)}\subseteq\P^{\Lambd}_{p,q,r}={\Gamma}_{p,q,r}^{(0)}$), $\check{\ad}$ ($\P^{\pm}_{p,q,r}=\check{\Gamma}_{p,q,r}^{(0)}\subseteq\P^{\pm\Lambd}_{p,q,r}=\check{\Gamma}_{p,q,r}^{(1)}$), and $\tilde{\ad}$ ($\P^{\pm\Lambda}_{p,q,r}=\tilde{\Gamma}^{(1)}_{p,q,r}\subseteq\P^{\Lambda}_{p,q,r}=\tilde{\Gamma}^{(0)}_{p,q,r}$) in Theorem~\ref{theorem_g(1)}.  
We also prove that some of these groups leave invariant the grade-$0$ and grade-$n$ subspaces and their direct sum under $\ad$ ($\P^{\pm\rad}_{p,q,r}={\Gamma}_{p,q,r}^{0n}={\Gamma}_{p,q,r}^{n}$ in the case of even $n$), under $\check{\ad}$ ($\P^{\pm\rad}_{p,q,r}=\check{\Gamma}_{p,q,r}^n$ in the case of odd $n$ and $\P^{\pm}_{p,q,r}=\check{\Gamma}_{p,q,r}^{0}\subseteq\P_{p,q,r}=\check{\Gamma}_{p,q,r}^{0n}$ in the case of arbitrary $n$), and under $\tilde{\ad}$ ($\P^{\pm\rad}_{p,q,r}=\tilde{\Gamma}^{n}_{p,q,r}$, $\P_{p,q,0}=\tilde{\Gamma}^{0n}_{p,q,0}$, and $\P^{\pm\rad}_{p,q,r}=\tilde{\Gamma}^{0n}_{p,q,r}$ in the case $r\neq0$) in Theorems \ref{gogn_pqr} and \ref{gon_pqr}. 
We study the Lie algebras of the introduced Lie groups and calculate their dimensions in Theorem~\ref{Lie_alg_th}. In future, we plan to study the relation between the results of this paper and such concepts as root systems and universal enveloping algebras. 
Note that the groups preserving the other fundamental subspaces under $\ad$, $\check{\ad}$, and $\tilde{\ad}$ differ significantly from the groups introduced in this paper (it can be seen in the particular case of the non-degenerate algebra \cite{GenSpin,OnInner}).  In the further research, we are going to consider the groups preserving the subspaces determined by the grade involution and the reversion
 under 
 the adjoint and twisted adjoint representations
 in the degenerate geometric algebras $\C_{p,q,r}$. Also we are going to
study normalized subgroups of these groups, which can be interpreted as generalizations of the spin groups in the degenerate case and can be used in applications.

The well-known Clifford and Lipschitz groups \cite{RA_Z,ABS,brooke_2,lg1,lounesto} preserve the grade-$1$ subspace under the adjoint and twisted adjoint representations respectively. The groups (\ref{our_groups_1})--(\ref{our_groups_2}) contain the degenerate Clifford and Lipschitz groups as subgroups and can be considered as their analogues. 
The groups (\ref{our_groups_1})--(\ref{our_groups_2}) are closely related to the higher-dimensional Heisenberg groups (see, for example, \cite{baker,hall}) in the cases of the low-dimensional algebras (Section \ref{section_examples}).
The introduced groups may be useful for applications in physics \cite{brooke_1,brooke_2,brooke_3,phys,hestenes}, engineering \cite{h3,jl1,la2_}, quantum mechanics and computing \cite{brooke_1,quan}, computer science \cite{cs1,jl1,b1,hd,hd2}, spinor image processing \cite{ip}, spinor neural networks \cite{hi1}, and other sciences.

\appendix
\section{Summary of notation}\label{appendix_A}

According to the reviewer's recommendation, we provide an overview of notation used throughout the paper in Table \ref{table_notation}. We write out the notation, its meaning, and the place where it is mentioned for the first time.

\begin{longtable}{|c|c|c|} 
\caption{Summary of notation}\label{table_notation} \\ \hline
Notation & Meaning & \parbox{2.2cm}{\begin{center}First mention\end{center}} \\ \hline
$\C_{p,q,r}$ & \parbox{5.2cm}{\begin{center}(Clifford) geometric algebra over the real $\BR^{p,q,r}$ or complex $\BC^{p+q,r}$ vector space\end{center}} & page \pageref{def_ga}\\ \hline
$\F$ & \parbox{5.2cm}{\begin{center}Field of real or complex numbers in the cases $\BR^{p,q,r}$ and $\BC^{p+q,r}$  respectively\end{center}} & page \pageref{def_F}  \\ \hline
$\Lambda_r$ & Grassmann subalgebra $\C_{0,0,r}$ & page \pageref{def_grass} \\ \hline
$\C^{0}$ & Subspace of grade $0$ & page \pageref{def_g0}\\ \hline
$\C^{k}_{p,q,r}$ & Subspace of fixed grade $k=1,\ldots,n$ & page \pageref{def_gk} \\ \hline
$\C^{0n}_{p,q,r}$ & \parbox{5.2cm}{\begin{center}Direct sum of $\C^{0}$ and $\C^{n}_{p,q,r}$\end{center}} & page \pageref{def_g0n} \\ \hline
$\widehat{U}$ & Grade involute of $U\in\C_{p,q,r}$ & page \pageref{def_grade_inv} \\ \hline
$\C^{(0)}_{p,q,r}$, $\C^{(1)}_{p,q,r}$ & Even and odd subspaces & formula (\ref{even_odd_subspaces}) \\ \hline
$\rad\C_{p,q,r}$ & Jacobson radical of $\C_{p,q,r}$ & page \pageref{def_rad} \\ \hline
$A^\times$ & \parbox{5.2cm}{\begin{center}Subset of all invertible elements of a set $A$\end{center}}  & page \pageref{def_A_inv} \\ \hline
$\ad$ & \parbox{5.2cm}{\begin{center}Adjoint representation ${\ad}_{T}(U)=TU T^{-1}$\end{center}} &  formula (\ref{ar})\\ \hline
$\check{\ad}$ & \parbox{5.2cm}{\begin{center}Twisted adjoint representation $\check{\ad}_{T}(U)=\widehat{T}U T^{-1}$\end{center}} & \parbox{2.2cm}{\begin{center}formula (\ref{twa1})  \end{center}} \\ \hline
$\tilde{\ad}$ & \parbox{5.2cm}{\begin{center}Twisted adjoint representation $\tilde{\ad}_{T}(U)=TU_0 T^{-1}+\widehat{T} U_1 T^{-1}$\end{center}} & \parbox{2.2cm}{\begin{center}formulas (\ref{twa2}), (\ref{twa22})\end{center}} \\ \hline
$\Z_{p,q,r}$ & Center of $\C_{p,q,r}$ & formula (\ref{Zpqr}) \\ \hline
$\P^{\pm}_{p,q,r}$ & Group $\C^{(0)\times}_{p,q,r}\cup\C^{(1)\times}_{p,q,r}$ & formula (\ref{P+-}) \\ \hline
$\P_{p,q,r}$ & Group $\P^{\pm}_{p,q,r}\Z^{\times}_{p,q,r}$ & formula (\ref{defpz}) \\ \hline
$\P^{\pm\Lambda}_{p,q,r}$ & Group $\P^{\pm}_{p,q,r}\Lambd^{\times}_r$ & \parbox{2.2cm}{\begin{center}formula (\ref{P_Lambd})\end{center}} \\ \hline
$\P^{\Lambda}_{p,q,r}$ & Group $\P_{p,q,r}\Lambd^{\times}_r$ & \parbox{2.2cm}{\begin{center}formula (\ref{pLambd=pLambd})\end{center}} \\ \hline
$\Gamma^{(k)}_{p,q,r}$, $\check{\Gamma}^{(k)}_{p,q,r}$, $\tilde{\Gamma}^{(k)}_{p,q,r}$ & \parbox{5.2cm}{\begin{center}Groups preserving the subspaces of fixed parity $\C^{(k)}_{p,q,r}$, $k=0,1$, under $\ad$, $\check{\ad}$, and $\tilde{\ad}$ respectively\end{center}} & \parbox{1.7cm}{\begin{center}formulas (\ref{g(k)_notation})--(\ref{g(k)_notation__})\end{center}}\\ \hline
$\Gamma^{k}_{p,q,r}$, $\check{\Gamma}^{k}_{p,q,r}$, $\tilde{\Gamma}^{k}_{p,q,r}$ & \parbox{5.2cm}{\begin{center}Groups preserving the subspaces of fixed grades $\C^{k}_{p,q,r}$, $k=0,1,\ldots, n$, under $\ad$, $\check{\ad}$, and $\tilde{\ad}$ respectively\end{center}} & \parbox{1.7cm}{\begin{center}formulas (\ref{g_k_1})--(\ref{g_k_3})\end{center}} \\ \hline
$\Gamma^{0n}_{p,q,r}$, $\check{\Gamma}^{0n}_{p,q,r}$, $\tilde{\Gamma}^{0n}_{p,q,r}$ & \parbox{5.2cm}{\begin{center}Groups preserving the subspace $\C^{0n}_{p,q,r}$ under $\ad$, $\check{\ad}$, and $\tilde{\ad}$ respectively\end{center}} & \parbox{2.2cm}{\begin{center}formulas (\ref{g_0n_1})--(\ref{g_0n_3})\end{center}}\\ \hline
\end{longtable}

\section*{Acknowledgments}

The main results of this paper were reported at the International Conference of Advanced Computational Applications of Geometric Algebra (ICACGA), Denver, USA, October 2022. The authors are grateful to the organizers and the participants of this conference for fruitful discussions.

The authors express their gratitude to the reviewers for their helpful comments on how to improve the paper. The authors are grateful to Prof. Leo Dorst for pointing out another definition of the twisted adjoint representation for elements of higher grades, which takes into account the signs of grades and is, therefore, more preferable when using the Cartan--Dieudonn\'e theorem (see Section \ref{the_kernels}).

The publication was prepared within the framework of the Academic Fund Program at HSE University in 2022 (grant 22-00-001).

\medskip

\noindent{\bf Data availability} Data sharing not applicable to this article as no datasets were generated or analyzed during the current study.

\medskip

\noindent{\bf Declarations}\\
\noindent{\bf Conflict of interest} The authors declare that they have no conflict of interest.

\bibliographystyle{spmpsci}

\begin{thebibliography}{100}

\bibitem{RA_Z} Ablamowicz, R.: Structure of spin groups associated with degenerate Clifford algebras. J. Math. Phys. \textbf{27}(1), 1--6 (1986)

\bibitem{RA_2} Ablamowicz, R., Lounesto, P.: Primitive Idempotents and Indecomposable Left Ideals in Degenerate Clifford Algebras. Clifford Algebras and Their Applications in Mathematical Physics, (1986)

\bibitem{ABS} Atiyah, M., Bott, R., Shapiro, A.: Clifford Modules. Topology \textbf{3}, 3--38 (1964)

\bibitem{baker}
Baker, A.: Matrix Groups: An Introduction to Lie Group Theory.  Springer, New York (2002)

\bibitem{ip} Batard, T., Berthier, M.: Clifford–Fourier Transform and Spinor Representation of Images.
Quaternion and Clifford Fourier Transforms and Wavelets, Trends in Mathematics \textbf{27}, (2013)

\bibitem{b1} Bayro-Corrochano, E.: Geometric algebra applications. Vol. I, Computer vision, graphics and neurocomputing. Springer, Cham, Switzerland (2019)

\bibitem{b2} Bayro-Corrochano, E.: Geometric Computing - for Wavelet Transforms, Robot Vision, Learning, Control and Action. Springer, London (2010)

\bibitem{bcsd} Bayro-Corrochano, E., Daniilidis, K., Sommer, G.: Motor Algebra for 3D Kinematics: The Case of the Hand-Eye Calibration. J. Math. Imaging Vis \textbf{13}, 79--100 (2000)

\bibitem{h3} Bayro-Corrochano, E., Sobczyk, G.: Geometric Algebra with Applications in Science and Engineering. Birkhäuser Boston, MA (2001)

\bibitem{lg1} 
Benn, I., Tucker, R.:
An introduction to Spinors and Geometry with Applications in Physics.
Bristol (1987)

\bibitem{brooke_1} Brooke, J.: A Galileian formulation of spin: I. Clifford algebras and spin groups. J. Math. Phys. \textbf{19}, 952--959 (1978)

\bibitem{brooke_2} Brooke, J.: Spin Groups Associated with Degenerate Orthogonal Spaces. Clifford Algebras and Their Applications in Mathematical Physics, Part of the NATO ASI Series \textbf{183}, 93--102 (1986)

\bibitem{brooke_3} Brooke, J.: Clifford Algebras, Spin Groups and Galilei Invariance - New Persectives. Thesis, U. of Alberta, (1980)

\bibitem{CJW2} Catto, S., Choun, Y., Gurcan, Y., Khalfan, A., Kurt, L.: Grassmann Numbers and Clifford-Jordan-Wigner Representation of Supersymmetry. J. of Physics: Conference Series, \textbf{411}, (2013)

\bibitem{Choi} Choi, H.I., Lee, D.S.,  Moon, H.P.: Clifford Algebra, Spin Representation, and Rational Parameterization of Curves and Surfaces. Advances in Computational Mathematics \textbf{17}, 5–48 (2002)

\bibitem{Chr} Chrysikos, I.: Dirac operators in geometry - Lecture Notes (2019). Available at \texttt{https://prf.uhk.cz/geometry/DiracNotesSchool2019Chrysikos.pdf}

\bibitem{crum_book} Crumeyrolle, A.: Orthogonal and Symplectic Clifford Algebras. 1st edn. Springer, Netherlands (1990)

\bibitem{crum} Crumeyrolle, A.: Algebres de Clifford degenerees et revetements des groupes conformes affines orthogonaux et symplectiques. Ann. Inst. H. Poincare \textbf{33}(3), 235--249 (1980)

\bibitem{Dai} Dai, X.: Lectures on Dirac Operators and Index Theory (2015). Available at \texttt{https://web.math.ucsb.edu/\~{}dai/book.pdf}

\bibitem{der} Dereli, T., Kocak, S., Limoncu, M.: Degenerate Spin Groups as Semi-Direct Products. Advances in Applied Clifford Algebras \textbf{20}, 565--573 (2010)

\bibitem{phys} Doran, C., Lasenby, A.: Geometric Algebra for Physicists. Cambridge University Press, Cambridge, UK (2003)

\bibitem{pga_book} Dorst, L., De Keninck, S.: A Guided Tour to the Plane-Based Geometric Algebra PGA. Version 2.0 (2022). Available at \texttt{http://bivector.net/PGA4CS.html}

\bibitem{jl1} Dorst, L., Doran, C., Lasenby, J.: Applications of Geometric Algebra in Computer Science and Engineering. Birkhäuser Boston, MA (2002)

\bibitem{cs1} Dorst, L., Fontijne, D., Mann, S.: Geometric Algebra for Computer Science: An Object-Oriented Approach to Geometry. Morgan Kaufmann Publishers Inc., Burlington (2007)

\bibitem{la2_} Dorst, L., Lasenby, J.: Guide to geometric algebra in practice. Springer, London (2011)

\bibitem{GenSpin} Filimoshina, E., Shirokov, D.: On generalization of Lipschitz groups and spin groups. Mathematical Methods in the Applied Sciences, 1--26 (2022)

\bibitem{ICACGA} Filimoshina, E., Shirokov, D.: On some Lie groups in degenerate geometric algebras. In: Eckhard
Hitzer \& Dietmar Hildenbrand (eds), First International Conference, ICACGA 2022, Colorado Springs, CO, USA, 2022, Proceedings. Lecture Notes in Computer Science. Springer, Cham, 2023 (to appear)

\bibitem{gunn_1} Gunn, C.: Geometric Algebras for Euclidean Geometry. Adv. Appl. Clifford Algebras \textbf{27}, (2017)

\bibitem{gunn_2} Gunn, C.: Doing euclidean plane geometry using projective geometric algebra. Adv. Appl. Clifford Algebras \textbf{27}, (2017)

\bibitem{hall} Hall, B.: Lie Groups, Lie Algebras, and Representations. Springer, Cham, Switzerland (2015)

\bibitem{Harvey}
Harvey, F.: Spinors and Calibrations. Academic Press (1990)

\bibitem{Helm} Helmstetter, J., Micali, A.: Quadratic Mappings and Clifford Algebras. Birkhäuser, Basel (2008)

\bibitem{hestenes} Hestenes, D., Sobczyk, G.: Clifford Algebra to Geometric Calculus - A Unified
Language for Mathematical Physics. Reidel Publishing Company, Dordrecht Holland (1984)

\bibitem{hestenes_CGA} Hestenes, D.: Old Wine in New Bottles: A New Algebraic Framework for Computational Geometry. Geometric Algebra with Applications in Science and Engineering, 3--17 (2001)

\bibitem{hi1} Hitzer, E.: Geometric operations implemented by conformal geometric algebra neural nodes. In: Proc. SICE Symposium on Systems and Information 2008, pp. 357--362. Himeji, Japan (2008), arXiv:1306.1358v1 

\bibitem{hd} Hildenbrand, D.: Foundations of Geometric Algebra Computing. 1st edn. Springer, Berlin, Heidelberg (2013)

\bibitem{hd2} Hildenbrand, D.: The Power of Geometric Algebra Computing. 1st edn. Chapman and Hall/CRC, New York (2021)

\bibitem{pga1} Hrdina, J., Navrat, A., Vasik, P., Dorst, L.: Projective Geometric Algebra as a Subalgebra of Conformal Geometric algebra. Adv. Appl. Clifford Algebras \textbf{31}(18), (2021)

\bibitem{Knus}  Knus, M.-A.: Quadratic and Hermitian forms over rings. Springer, Berlin, p.~228 (1991)

\bibitem{JR_1} Lam, T.: A First Course in Noncommutative Rings, Graduate Texts in Mathematics. Vol. 131, 2 edn. Springer-Verlag (2001)

\bibitem{h2} Li, H., Hestenes, D., Rockwood, A.: Generalized Homogeneous Coordinates for Computational Geometry. Geometric Computing with Clifford Algebras, Springer, Berlin, Heidelberg, 27--59 (2001)

\bibitem{lounesto} Lounesto, P.: Clifford Algebras and Spinors. Cambridge University Press, Cambridge, UK (1997)

\bibitem{LuSv} Lundholm, D., Svensson, L.: Clifford algebra, geometric algebra, and applications (2009), arXiv:0907.5356

\bibitem{p} Porteous, I.: Clifford Algebras and the Classical Groups. Cambridge University Press, Cambridge, UK (1995)

\bibitem{se} Selig, J., Bayro-Corrochano, E.: Rigid body dynamics using Clifford algebra.
Advances in Applications of Clifford algebras \textbf{20}, 141–-154 (2010)

\bibitem{it5}
Shirokov, D.:  Calculation of elements of spin groups using generalized Pauli's theorem. Adv. Appl. Clifford Algebras \textbf{25}(1), (2015)

\bibitem{it4}
Shirokov, D.: Calculation of elements of spin groups using method of averaging in Clifford's geometric algebra. Adv. Appl. Clifford Algebras \textbf{29}(50), (2019)

\bibitem{OnInner} Shirokov, D.: On inner automorphisms preserving fixed subspaces of Clifford algebras. Adv. Appl. Clifford Algebras \textbf{31}(30), (2021)  

\bibitem{it1}
Shirokov, D.: Symplectic, Orthogonal and Linear Lie Groups in Clifford Algebra. Adv. Appl. Clifford Algebras \textbf{25}(3), 707--718 (2015)

\bibitem{it2}
Shirokov, D.: On Some Lie Groups Containing Spin Group in Clifford Algebra. Journal of Geometry and Symmetry in Physics \textbf{42}, 73--94 (2016)

\bibitem{it3}
Shirokov, D.: Classification of Lie algebras of specific type in complexified Clifford algebras. Linear and Multilinear Algebra \textbf{66}(9), 1870--1887 (2018)


\bibitem{quan}
Tolar, J.: On Clifford groups in quantum computing. Journal of Physics: Conference Series \textbf{1071}, (2018)

\bibitem{Wal} Walpuski, T.: Differential Geometry IV, Lecture Notes (2022). Available at 
\texttt{https://walpu.ski/Teaching/SS22/DifferentialGeometry4/}

\bibitem{la1_} Wareham, R., Cameron, J., Lasenby, J.: Applications of Conformal Geometric Algebra in Computer Vision and Graphics. Lecture Notes in Computer Science \textbf{3519}, 329--349 (2004)



\bibitem{Zer} Zerouali, A.: Twisted Conjugation on Connected Simple Lie Groups and Twining Characters (2020), 	arXiv:1811.06507


















\end{thebibliography}

\end{document}